\documentclass[11pt]{amsart}
\pdfoutput=1
\usepackage[utf8]{inputenc}
\usepackage{amsmath, amssymb}
\usepackage[english]{babel}
\usepackage{graphicx}
\usepackage{amsthm}
\usepackage{array}
\usepackage{amsthm}
\usepackage{mathtools}
\usepackage{thmtools}
\usepackage{lpic}
\usepackage[margin=.85 in]{geometry}
\allowdisplaybreaks
\usepackage{enumitem}

\usepackage[dvipsnames]{xcolor}
\usepackage{hyperref}
\definecolor{candyapplered}{rgb}{1.0, 0.03, 0.0}
\definecolor{mediumblue}{rgb}{0.0, 0.0, 0.8}
\hypersetup{
pdfauthor={Jonah Gaster},
pdftitle={},
bookmarksnumbered,
colorlinks=true,
linkcolor=mediumblue,
citecolor=candyapplered,
urlcolor=blue}

\usepackage{caption,subcaption}
\usepackage[ocgcolorlinks]{ocgx2} 
\declaretheorem[numberwithin=section]{theorem}
\declaretheorem[sibling=theorem, style=definition]{definition}
\declaretheorem[sibling=theorem]{lemma}
\declaretheorem[sibling=theorem, style=remark]{remark}
\declaretheorem[sibling=theorem]{proposition}
\declaretheorem[sibling=theorem]{corollary}

\declaretheorem[sibling=theorem, style=remark]{example}

\declaretheorem[sibling=theorem, style=remark]{problem}

\numberwithin{equation}{section}

\newcommand{\bZ}{\mathbb{Z}}
\newcommand{\bP}{\mathbb{P}}
\newcommand{\bN}{\mathbb{N}}
\newcommand{\bH}{\mathbb{H}}

\newcommand{\bR}{\mathbb{R}}
\newcommand{\bQ}{\mathbb{Q}}
\newcommand{\bC}{\mathbb{C}}

\newcommand{\cP}{\mathcal{P}}

\newcommand{\cT}{\mathcal{T}}
      
\newcommand{\SL}{\mathrm{SL}}
\newcommand{\PSL}{\mathrm{PSL}}
\newcommand{\cF}{\mathcal{F}}
\newcommand{\cH}{\mathcal{H}}
\newcommand{\cM}{\mathcal{M}}

\newcommand{\cS}{\mathcal{S}}
\newcommand{\cY}{\mathcal{Y}}
\newcommand{\ML}{\mathcal{ML}}
\newcommand{\PML}{\mathbb{P}\mathcal{ML}}

\newcommand{\tw}{\mathrm{tw}}

\newcommand{\X}{\mathfrak{X}}
\newcommand{\Isom}{\mathrm{Isom}}

\mathtoolsset{showonlyrefs=true}

\begin{document}

\title[Twist numbers]{Twist Numbers on Hyperbolic One-Holed Tori}
\author[Gaster]{Jonah Gaster}
\date{}

\address{Department of Mathematical Sciences, University of Wisconsin-Milwaukee}
\email{gaster@uwm.edu}


\begin{abstract}
On a hyperbolic surface homeomorphic to a torus with a puncture,
each oriented simple geodesic inherits a well-defined relative twist number in $[0,1]$, given by the ratio to its hyperbolic length of the hyperbolic distance between the orthogonal projections of the cusp (or boundary) on its left and right, respectively.
With the Markov Uniqueness Conjecture in mind, the twist numbers for simple geodesics on the modular torus $\X$ are of particular interest.
Up to isometries of $\X$, simple geodesics are parameterized naturally by $\bQ\cap[0,1]$, and the relative twist number yields a map $\tau_\X:\bQ\cap [0,1] \to [0,1]$. 
We use hyperbolic geometry and the Farey graph to show that the graph of $\tau_\X$ is dense in $[0,1]\times [0,1]$, 
and the same conclusion holds for any complete hyperbolic structure $\cY$ on the punctured torus.
It follows that the twist number of a simple closed curve on the punctured torus does not extend continuously to the space of measured laminations.
We also include some explicit calculations of geometric quantities associated to Markov triples, and the curious fact that $\tau_\X$ is never equal to zero.
\end{abstract}

\maketitle

\begin{figure}[h!]
\centering
\captionsetup{width=.8\linewidth}
\includegraphics[width=12cm]{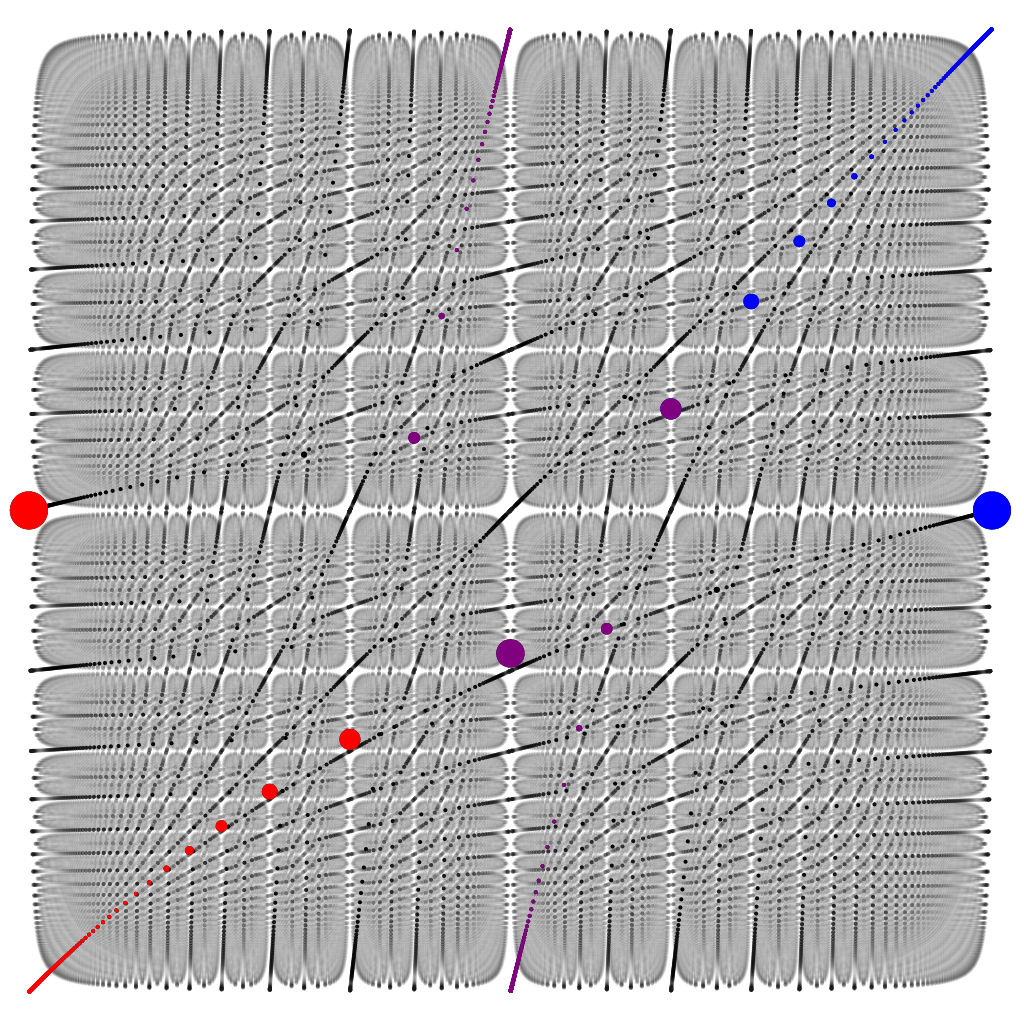}
\caption{The graph of $\tau_\X$ on points in $\bQ$ determined by the $1216589$ Markov numbers whose Farey labels have denominator $\le 2000$. 
}
\label{pic: twist graph}
\end{figure}

\section{Introduction}

The \emph{modular torus} is the hyperbolic surface $\X:=\bH^2/\Gamma$, where $\Gamma$ is the commutator subgroup of $\PSL(2,\bZ)$. 
The modular torus is finite-volume, homeomorphic to a punctured torus, with isometry group the dihedral group with twelve elements.
It has been known for many years that the modular torus is intimately connected to numbers of broad arithmetic interest
\cite{Aigner,Baragar1,Beardon-Lehner-Sheingorn,Cohn,Goldman,Haas,Lehner-Sheingorn,Matheus,McShane-new,Schmutz,Series,Sheingorn,Springborn1,Springborn2}
owing to the identification of the relative $\PSL(2,\bC)$-character variety of the punctured torus with the locus of triples $(X,Y,Z)\in \bC^3$ with
\begin{equation}
\label{eq:characters}
X^2+Y^2+Z^2=X Y Z~,
\end{equation}
where $X$, $Y$, and $Z$ represent traces of a trio of simple closed curves that pairwise intersect once.
Integer solutions to \eqref{eq:characters} represent characters of $\X$.
The \emph{Markov numbers} $\cM=\{1,2,5,\ldots\}$ are positive integers that appear in a Markov triple \cite{Markov}, i.e.~a solution $(x,y,z)\in\bN^3$ to the cubic equation
\begin{equation}
\label{eq:markov}
x^2+y^2+z^2=3xyz~.
\end{equation}

The \emph{Markov Uniquness Conjecture} (MUC), first formulated by Frobenius \cite{Frobenius}, asserts that each positive Markov triple is determined by its largest entry.
An equivalent version of MUC is the following natural assertion about the simple length spectrum of $\X$: \emph{If two simple closed geodesics on $\X$ have the same length, then they lie in a common $\Isom(\X)$-orbit.} 
In other words, there are no unexplained coincidences in the simple length spectrum of the modular torus.

Let $\cS$ indicate the collection of simple closed geodesics on $\X$.
In the above reformulation, the length of $\gamma\in \cS$ plays the role of a Markov number $m\in \cM$, which is quite natural as equations \eqref{eq:characters} and \eqref{eq:markov} indicate that the trace of a simple closed geodesic on $\X$ is exactly three times a Markov number. 
The purpose of this article is to investigate another  invariant, naturally arising from the viewpoint of the hyperbolic geometry of $\X$, associated to $\gamma\in\cS$: the \emph{relative twist number} $\tau_\X(\gamma)$.

The completion of the complement of a simple closed geodesic on a finite-volume hyperbolic punctured torus is a pair of pants $\cP_\gamma$ with two geodesic boundary components and one cusp. 
Such a hyperbolic surface with boundary is determined up to isometry by the lengths of the boundary components \cite{Farb-Margalit}. 
Therefore, if $\gamma,\eta\in\cS$ have equal length, we find an isometry $\phi:\cP_\gamma \approx \cP_\eta$.
MUC concerns the existence of an extension of $\phi$ to $\bar{\phi}:\X\to\X$, obtained by gluing $\phi$ together along the pair of geodesic boundary components of $\cP_\gamma$ and $\cP_\eta$.

This extension problem is classical in Teichm\"uller theory, as it is 
core to the construction of Fenchel-Nielsen coordinates, an identification of the Teichm\"uller space with a cell. 
The existence of $\bar{\phi}$ is equivalent to the equality $\tau_\X(\gamma)=\tau_\X(\eta)$, where
$\tau_\X(\gamma)$ is constructed as follows: 
Endow $\gamma$ with an orientation.
There are a pair of orthogeodesics from the geodesic boundary of $\cP_\gamma$ to the cusp. 
The pair of intersection points on $\gamma$ are the orthogonal projections of the cusp to $\gamma$ from its left and right, and one defines $\tau_\X(\gamma)$ to be the normalized distance from the left point to the right, along the positive orientation of $\gamma$ (see \S\ref{subsec: geom background} for detail).
It is not hard to see that $\tw_\X(\gamma)$ is independent of orientation, and we find that
\[
\Isom(\X)\cdot \gamma =\Isom(\X)\cdot \eta \ \ \iff \ \ (\ell_\X(\gamma), \tau_\X(\gamma))=(\ell_\X(\eta),\tau_\X(\eta)) ~.
\]
It is straightforward how to alter the above definition of twist number for any complete hyperbolic structure on the punctured torus, finite-volume or not; see \S\ref{sec: other structures}.

The analysis of the length spectrum of $\cS$ is facilitated by a certain function due to V.~Fock, a comparison of `arithmetic' versus `hyperbolic' length for simple geodesics on $\X$. 
We briefly describe Fock's construction:
Simple geodesics on punctured tori are naturally parameterized by lines in the integral homology, or, equivalently, by their slope, a reduced fraction in $\bQ\cup \{\infty\}$. 
When this identification is made so that the three systoles of $\X$ are $0$, $\infty$, and $-1$, a fundamental domain for the action of $\Isom(\X)$ is given by $\bQ\cap[0,1]$.
The \emph{Fock function} $\Psi: [0,1]\to \bR$ is a continuous, decreasing, convex function, whose value at $\frac pq\in[0,1]$ is given by
\begin{equation}
\label{eq:Fock}
\Psi\left( \frac pq\right) = \frac{\ell_\X\left( \frac pq \right) }{q}~,
\end{equation}
where $\ell_\X(p/q)$ is the hyperbolic length of the simple geodesic corresponding to $p/q$ on $\X$.
The convenient properties of Fock's function -- in particular, the existence of a continuous extension from $\bQ\cap [0,1]$ to $[0,1]$ -- indicate that the parameterization of $\cS$ by slopes is somehow `consistent' with the hyperbolic length function $\ell_\X:\cS\to\bR$.

With MUC in mind, one could reasonably ask about a version of Fock's function for the relative twists $\tau_\X$; roughly, this question asks about the `consistency' of the parameterization $\cS\approx \bQ\cup\{\infty\}$ with the relative twist numbers.
In this note, we will prove:

\begin{theorem}
\label{main thm}
The graph of $\tau_\X:\bQ\cap[0,1] \to [0,1]$ is dense in $[0,1]\times[0,1]$.
\end{theorem}
\noindent That is, there is no hope for an extension of $\tau_\X$ from $\bQ\cap[0,1]$ to $[0,1]$, and $\cS\approx \bQ\cup\{\infty\}$ and $\tau_\X$ are not really very consistent at all.

This theorem sheds some light on why MUC is difficult: 
the conjecture is equivalent to the claim that each Markov number has a well-defined relative twist number. 
One might hope to gain some control on multiplicity with respect to MUC by leveraging the twist number, but the determination of $\tau_\X(p/q)$ from $p/q\in\cS$ is a real stumbling block.
As this function is quite complicated, it is unclear how to glean information about the twist numbers from the value of the corresponding Markov number alone.

For the proof of Theorem~\ref{main thm}, we relate the relative twist $\tau_\X$ to an `arithmetic' version, the \emph{Farey twist} function $\tau_\cF:\bQ\cap[0,1]\to \bQ\cap [0,1]$ that arises in consideration of the Farey graph (see Definition~\ref{def: Farey twist}). We are unaware whether consideration of $\tau_\cF$ has appeared in the literature; it is of natural geometric interest, so its study in the geometry of numbers would seem plausible. 
We provide the estimate:

\begin{proposition}
\label{prop:comparing twists}
There is a constant $C>0$ so that, for every $p/q\in\bQ\cap[0,1]$,
\[
\left | \tau_\X\left( \frac pq \right) - \tau_\cF \left( \frac pq\right) \right| < \frac C{q^2}~.
\]
\end{proposition}

Roughly speaking, this Proposition is a result of the following geometric intuition: 
There is a correspondence between Markov triples and Farey triples, and one finds that
each triple corresponds to a trio of simple geodesics on $\X$ that pairwise intersect once, with a complementary hyperbolic triangle with side lengths given by the half-lengths of the three geodesics.
As $q\to\infty$, one proceeds `down' the Farey graph, these hyperbolic triangles flatten out, and the hyperbolic lengths become roughly additive. 
The relative twist $\tau_\X(p/q)$ can then be characterized as roughly the ratio between a pair of these half-lengths, which is roughly the ratio of corresponding Farey denominators, i.e.~$\tau_\cF(p/q)$; 
see \S\ref{sec: comparison} for detail.

In \S\ref{sec: Farey analysis}, we show that the Farey twist function has dense graph:

\begin{theorem}
\label{thm: prescribe arith twist}
The graph of $\tau_\cF:\bQ\cap[0,1]\to[0,1]$ is dense in $[0,1]\times[0,1]$.
\end{theorem}

See Figures~\ref{pic: F twist graph} and \ref{pic: twist graph} for images of the graphs of $\tau_\X$ and $\tau_\cF$ on the $\approx 1.2$ million points $\frac pq\in\bQ\cap[0,1]$ with $q\le 2000$.

Together, Proposition~\ref{prop:comparing twists} and Theorem~\ref{thm: prescribe arith twist} immediately imply Theorem~\ref{main thm}: it suffices to demonstrate that, for any $x\notin \bQ$, the point $(x,y)\in[0,1]^2$ is an accumulation point of the graph of $\tau_\X$. 
By Theorem~\ref{thm: prescribe arith twist} there is a sequence $\left(\frac{p_n}{q_n},\tau_\cF\left(\frac{p_n}{q_n}\right)\right)\to (x,y)$. Because $x\notin\bQ$, we have $q_n\to \infty$, so Proposition~\ref{prop:comparing twists} implies that $\tau_\X\left(\frac{p_n}{q_n}\right)\to y$ as well.

As for other hyperbolic structures $\cY$ on the punctured torus, we may define \emph{mutatis mutandis} the analogous twist function $\tau_\cY:\bQ\cap[0,1]\to[0,1]$. 
Essentially the same proof goes through (see \S\ref{sec: other structures}).

\begin{theorem}
\label{thm: other points}
Let $\cY$ be a complete hyperbolic surface homeomorphic to a punctured torus. Then the graph of $\tau_\cY:\bQ\cap[0,1]\to[0,1]$ is dense in $[0,1]\times[0,1]$.
\end{theorem}

Recall that weighted simple closed geodesics $\cS\times \bR_{\ge 0}$ are dense in the space $\ML(S)$ of measured laminations on $S$, and $\cS$ is dense inside the space $\PML(S)$ of projective measured laminations \cite{Thurston}.

\begin{corollary}
\label{cor: no continuous extension}
For any hyperbolic structure $\cY$ on the punctured torus, the relative twist map from $\cS$ to $\bR$ admits no continuous extension to $\PML(S)$. Similarly, the absolute twist map 
$\overline{\tw}_\cY:\cS\times\bR_{\ge 0} \to \bR$, given by $\overline{tw}_\cY(\gamma,t)=t\cdot \tw_\cY(\gamma)$,
admits no continuous extension to $\ML(S)$. 
\end{corollary}

\begin{proof}
Observe first that the second statement follows from the first. Indeed, if $\overline{\tw}_\cY:\ML(S)\to\bR$ is continuous, then,
because the hyperbolic length $\ell_\cY:\ML(S)\setminus\{0\}\to\bR$ is positive and homogeneous, 
one may form the well-defined map 
\[
\overline{\tau}_\cY=\frac{\overline{\tw}_\cY}{\ell_\cY}:\PML(S)\to \bR~,
\]
a continuous extension of $\tau_\cY:\cS\to \bR$.
As for the first statement, such an extension cannot exist by Theorem~\ref{thm: other points}, as the parameterization of $\cS\subset \PML(S)$ by slopes on the punctured torus provides a homeomorphism to $\bQ\bP^1\subset\bR\bP^1$ \cite{Hatcher-laminations}.
\end{proof}

\begin{remark}
\label{rem: david}
This corollary can be appreciated in light of Thurston's work on length functions. 
Thurston showed that the homogeneous extension of the length function $\ell_\cY:\cS\times\bR_{\ge0}\to \bR$, 
with $\overline{\ell}_\cY(\gamma,t)=t\cdot\ell_\cY(\gamma)$, extends continuously to $\ML(S)$ \cite[Ch.~9]{Thurston}.
In fact, `length functions' on $\cS$ are now known to extend continuously to $\ML(S)$ in great generality \cite{Bonahon, Otal,Erlandsson-Uyanik, Martinez-Granado} (but see also \cite[p.~4]{erlandsson-parlier-souto} and \cite[Prop.~11]{bonahon2} for other negative results).
Of course, there can be no continuous extension to $\PML(S)$ of the length function $\ell_\cY:\cS\to \bR$ for a hyperbolic metric $\cY$ on the punctured torus, since curves that accumulate on an irrational lamination have hyperbolic length going to infinity.
Nevertheless, Fock's function $\Psi$ 
corrects this problem: a continuous extension can be arranged by weighting against the arithmetic height of the curve.

It is tempting to imagine that one \emph{can} define an absolute twist number for a measured geodesic lamination on a hyperbolic punctured torus by integrating a `leafwise twist' against the transverse measure.
Corollary~\ref{cor: no continuous extension} indicates that such a definition has inherent challenges.
\end{remark}

\begin{remark}
\label{rem: symplectic}
Theorem~\ref{thm: other points} deserves comparison as well with the setting in which the curve $\gamma$ is fixed. Wolpert showed that, when $\{\gamma_i\}$ are the curves in a pants decomposition for $S$, the functions $\ell_\cdot(\gamma_i)$ and $\tau_\cdot(\gamma_i)$ on the Teichm\"uller space $\cT(S)$ form Darboux coordinates for $\cT(S)$, with respect to the Weil-Petersson symplectic form $\omega_\mathrm{WP}$ on $\cT(S)$ \cite{Wolpert}. It is striking that, by contrast, $\ell_\X$ and $\tau_\X$, as functions on the set of simple closed curves $\cS$, have such starkly different natures.
\end{remark}

It might be interesting to investigate what can be said about relative twist numbers in higher genus. 
However, it is not clear how the relative twist should be defined, even if one restricts to once-punctured hyperbolic surfaces $\cY$ and nonseparating curves $\gamma$: in $\cY\setminus \gamma$, projections of the cusp along simple orthogeodesics form a non-discrete set.
The punctured torus is small enough that there are precisely two such projections, so the twist numbers arise naturally.

Finally, in \S\ref{sec: explicit} we exploit the geometric work in \S\ref{sec: background} and \S\ref{sec:compute geom} to perform explicit calculations of geometric quantities associated to a Markov triple (as opposed to a single Markov number), including the relative twist. 
Namely, we show the amusing consequence that no simple geodesic on $\X$ has zero twist (see Corollary~\ref{cor: nonzero}), and that the Markov numbers obtained by twisting around the largest entry $n$ in a Markov triple $(n_1,n,n_2)$ are encoded as lifts of Weierstrass points along a lift of the simple geodesic corresponding to $n$ (see Proposition~\ref{prop: neighbors}).

\subsection{Acknowledgements} The author thanks Hugo Parlier for a conversation that contributed to the writing of this paper, and David Dumas, Caglar Uyanik, and D\'idac Martinez-Granado for helpful feedback.
Figures~\ref{pic: F twist graph} and~\ref{pic: twist graph} were made with the aid of Mathematica \cite{Mathematica}.

\bigskip
\section{The setting}
\label{sec: background}

We collect notation and some details for our analysis of the geometry of $\X$. Core to the story of the Markov numbers is the Farey graph. In this section, we identify several `labellings' of the Farey graph, which we are lead by hyperbolic geometry to interpret as maps to $\bR$ from $\cH$, a collection of horoballs 
from the maximal $\PSL(2,\bZ)$-invariant horoball packing of $\bH^2$.

\subsection{The Farey graph}
Recall that the Farey graph $\cF$ is the graph with vertex set $\bQ\cup \{\infty\}$, and where $\frac pq \sim \frac ab$ provided $|pb-qa|=1$. We adopt the convention that $\infty=1/0$, and that all other vertices of $\cF$ are written with positive denominator. 
This identifies the vertices of $\cF$ with primitive integral points in $\bZ_{\ge0}\times \bZ$.

\begin{figure}[h]
\centering
\includegraphics[width=6cm]{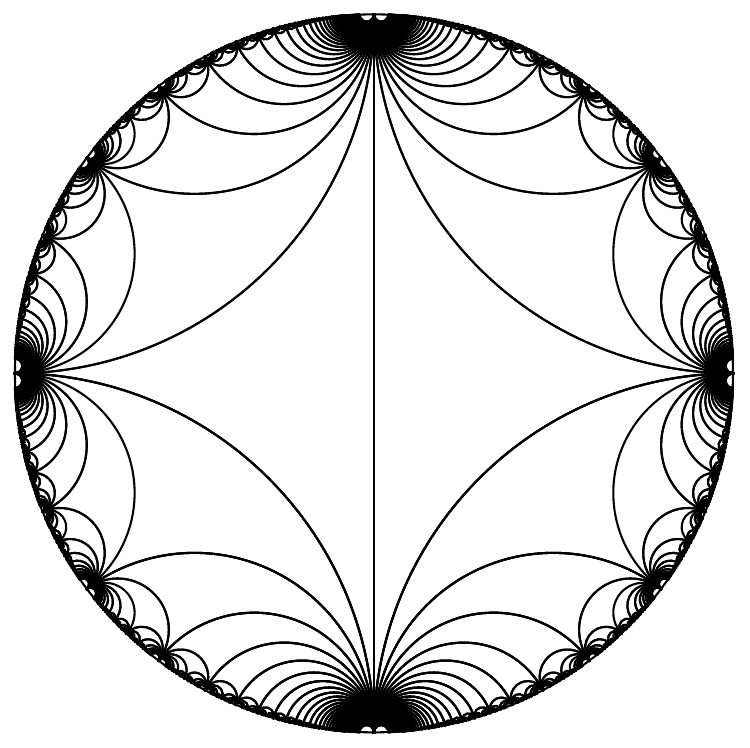}
\caption{The Farey graph embedded in $\bH^2\cup\partial_\infty\bH^2$.}
\label{pic:Farey}
\end{figure}

\begin{definition}[Farey order]
\label{def: Farey order}
The relation $\prec$ on $\cF$ is the transitive relation generated by requiring that $\frac pq \prec \frac ab$ whenever $\frac pq \sim \frac ab$ and $q\le b$.
\end{definition}
\noindent For instance, the reader can check that $\frac 13\prec \frac 25$, but $\frac 13 $ and $\frac 35$ are incomparable.
It is immediate that $\prec$ is a strict partial order with the property that intervals $\{x:x\prec y\}$ are finite.

Suppose that $p/q\in\cF$. B\'ezout's theorem implies that the neighbors of $p/q$ in $\cF$ are given by $\left\{ \frac{a+np}{b+nq} :n\in\bZ\right\}$, where $a/b$ is any Farey neighbor of $p/q$. 
It is easy to see that this implies that the clique number of $\cF$ is three, and that cliques have the convenient form $\frac ab< \frac pq < \frac cd$ with $a+c=p$ and $b+d=q$. 
We call such a trio a \emph{Farey triple} with \emph{center} $\frac pq$.
Moreover, there is a unique Farey triple $\frac ab< \frac pq < \frac cd$ with center $p/q$.
Consequently, $\frac ab$ and $\frac cd$ are the \emph{immediate precendents} of $\frac pq$, in that $\frac rs \prec \frac pq$ implies that either $\frac rs \preceq \frac ab$ or $\frac rs \preceq \frac cd$. 
When it is useful to keep track of the order induced by $\bQ\subset \bR$, we will refer to the immediate precedents as the ordered pair $\left( \frac ab, \frac cd\right)$.
Likewise, one may check that $\frac pq$ has \emph{immediate successors} $\left( \frac{a+p}{b+q} , \frac{p+c}{q+d}\right)$,
in that $\frac pq \prec \frac rs$ implies $ \frac{a+p}{b+q} \preceq \frac rs$ or $\frac{p+c}{q+d} \preceq \frac rs$.

As for comparing $\prec$ with the linear order of $\bQ\subset\bR$, it is easy to check:

\begin{lemma}
\label{lem: orders}
Suppose that $\frac pq$ has immediate precedents $\left( \frac ab,\frac cd\right)$. 
If $\frac pq \prec \frac rs$, then $\frac ab < \frac rs < \frac cd$.
\end{lemma}

The Farey graph admits a map to the compactification of the hyperbolic plane $\bH^2\cup\partial_\infty \bH^2$ so that 
its vertices $\bQ\cup\{\infty\}$ are sent to $\partial_\infty \bH^2$ by the identity map and so that its edges are sent to complete geodesics in $\bH^2$ -- here we assume the upper-half plane model of the hyperbolic plan $\bH^2$, in which one has the identification $\partial_\infty \bH^2 = \bR\cup \{\infty\}$. 
The characterization of neighbors of $p/q$ above can be used to show that the induced map $\cF\to \bH^2\cup \partial_\infty \bH^2$ is an embedding, in which complementary components of $\cF\cap \bH^2$ are ideal triangles.
This embedding will be implicit below, so that we may refer to $\cF\subset \bH^2\cup \partial_\infty \bH^2$.

Among its various fascinating aspects, the Farey graph arises as the `curve complex' of the torus (or punctured torus), an infinite-diameter $\delta$-hyperbolic metric space, and carries an isometric transitive action of $\PSL(2,\bZ)$ by linear fractional transformations. 
The interested reader can refer to \cite{Hatcher, Schleimer, Minsky, Masur-Minsky} for more of this fascinating story, and its deep connections to and profound consequences for the geometry and topology of surfaces and three-manifolds.

\subsection{The Farey labelling}
For our purposes, it will be useful to visualize the Farey graph from the viewpoint of the dual graph $\cT$ of $\cF \cap \bH^2$, in which there is a vertex for each complementary ideal triangle, with edges corresponding to a shared side of two triangles. 
By choosing barycenters for triangles and geodesic segments for edges, we will also view $\cT$ as embedded in $\bH^2$, transverse to $\cF$ (in fact, orthogonal to $\cF$).
Each component of $\bH^2\setminus \cT$ is within a bounded distance of a horoball (making it an `approximate horoball'), and is bounded by a complete piecewise geodesic. 
It will be convenient to associate to each complementary component the maximal horoball it contains, tangent to edges of $\cT$.
The map from a complementary region to the center of the corresponding horoball in $\bQ\cup\{\infty\}$ provides a bijection between the components of $\bH^2\setminus \cT$ and the vertices of $\cF$.

Let $\cH$ indicate the set of bounded maximal horoballs in $\bH^2\setminus \cT$ that intersect the strip $\{z\in \bH^2:0< \Re(z)<1\}$. 
Restricting the above correspondence, we obtain the \emph{Farey labelling}, evidently a bijection 
\[
\lambda_\cF : \cH \to \bQ\cap [0,1]
\]
that sends the maximal horoball $H$ to its center. 
Moreover, the strict partial order $\prec$ on $\bQ\cap[0,1]$ induced by $\cF$ induces a strict partial order $\prec$ on $\cH$.
The reader may check that $H\prec K$ precisely when there is a chain of consecutively tangent horoballs $H=H_0, H_1,\ldots, H_n=K$ in $\cH$ so that the Euclidean radius of $H_{i+1}$ is no larger than that of $H_i$.

Shortly, we will provide a `Markov labelling' of $\cH$, which is best described inductively. In analogy with that construction, we point out that one can describe $\lambda_\cF$ inductively as well: if the immediate precedents of $H\in \cH$ are $H_1$ and $H_2$, $\lambda_\cF(H_1)=\frac{p_1}{q_1}$ and $\lambda_\cF(H_2)=\frac{p_2}{q_2}$, then $\lambda_\cF(H)=\frac{p_1+p_1}{q_1+q_2}$. Because intervals $\{H:H\prec K\}$ are finite, $\lambda_\cF$ is induced by fixing labels $\frac 01$ and $\frac 11$ at two initial horoballs, and using the above law to produce the rest. See Figure~\ref{pic:Farey labels} for a schematic.

\begin{figure}
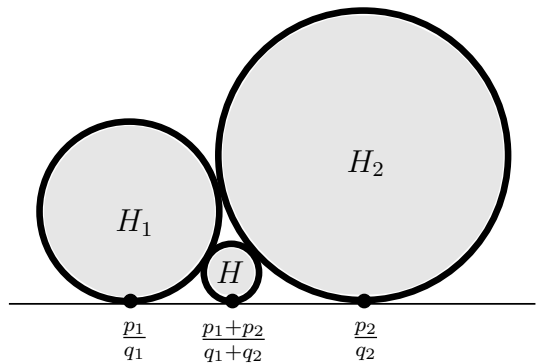

\begin{lpic}{Fareylabels(7cm)}
\large
\lbl[]{19,-5;$\frac {p_1}{q_1}$}
\lbl[]{34,-5;$\frac {p_1+p_2}{q_1+q_2}$}
\lbl[]{54,-5;$\frac{p_2}{q_2}$}
\lbl[]{19,13;$H_1$}
\lbl[]{33.5,5.5;$H$}
\lbl[]{54,22;$H_2$}
\end{lpic}
\vspace{1cm}
\caption{The inductive step in the Farey labelling of $\cH$, where $\lambda_\cF(H_1)=\frac{p_1}{q_1}$, $\lambda_\cF(H_2)= \frac{p_2}{q_2}$, and $\lambda_\cF(H)=\frac{p_1+q_1}{p_2+q_2}$. Notice that $\tau_\cF\left( \frac{p_1+p_2}{q_1+q_2}\right) = \frac{q_1}{q_1+q_2}$.}
\label{pic:Farey labels}
\end{figure}

\subsection{Farey twist}
We now define an arithmetic version of relative twist numbers for geodesics on $\X$.

\begin{definition}[Farey twist]
\label{def: Farey twist}
Suppose the immediate precedents of $\frac pq\in\bQ$ are $\left( \frac ab, \frac cd \right)$. 
The function $\tau_\cF:\bQ\cap[0,1]\to\bQ\cap[0,1]$ is given by 
\[
\tau_\cF\left(\frac pq \right) = \frac bq~.
\]
\end{definition}

For instance, because $\frac 01 < \frac1n< \frac 1{n-1}$ is a Farey triple, we find that 
$\tau_\cF\left( \frac 1n\right) = \frac 1n\to 0$ as $n\to \infty$. 
Similarly, the Farey triples $\frac{n-1}n<\frac n{n+1} < \frac 11$ demonstrate that 
$\tau_\cF\left( \frac n{n+1}\right) = \frac n{n+1}\to 1$
as $n\to\infty$.

\subsection{The Markov labelling}
The problem of enumerating $\cM$ lies at the heart of MUC.
Given a Markov triple $(x,y,z)$, one obtains another Markov triple via the \emph{Viet\'a involution}
\[
(x,z,y) \mapsto (x,3xy-z,y)~.
\]
Permuting the coordinates, there are two other Viet\'a involutions, and one can check that (aside from Markov triple $(1,1,1)$) one of these involutions will decrease the largest entry of $(x,y,z)$ while the other two will increase it.

We can interpret this graphically, as follows:
Let $\cT_0$ be a planar tree, with exactly two leaves at distance two, all of whose non-leaf vertices are trivalent. There is an embedding of $\cT_0$ in $\cT$ obtained by intersecting $\cT$ with the strip $\{z\in\bH^2: 0\le \Re(z) \le 1\}$.
Discarding the unbounded component, the remaining complementary components are again in correspondence with the horoballs $\cH$.

\begin{figure}[h]
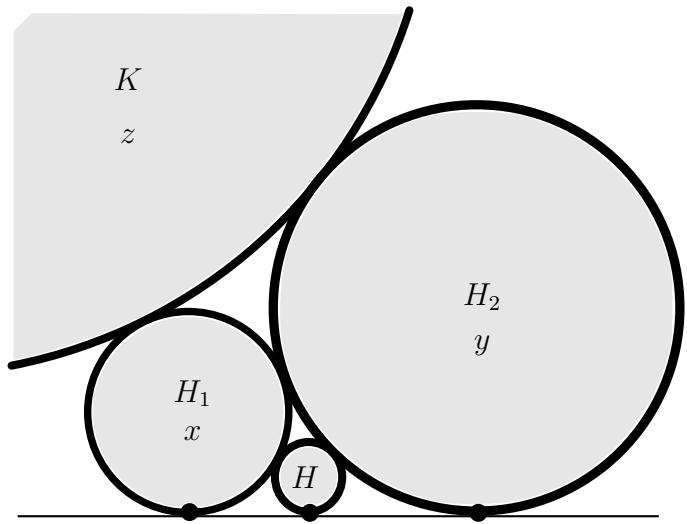

\begin{lpic}{MarkovLabels(9cm)}
\large
\lbl[]{15,55;$K$}
\lbl[]{15,48;$z$}
\lbl[]{23,16;$H_1$}
\lbl[]{23,11;$x$}
\lbl[]{59,28;$H_2$}
\lbl[]{59,22;$y$}
\lbl[]{37,5.5;$H$}
\end{lpic}
\caption{$\lambda_\cM(H)$ is obtained by applying the Viet\'a involution $\lambda_\cM(H)=3xy-z$.}
\label{pic: Markov labels}
\end{figure}

Now we initialize the Markov labelling $\lambda_\cM$. 
Put the label $1$ on each of the maximal horoballs, complementary to $\cT$, centered at $\infty$, $0$, and $-1$. 
Because intervals $\{H:H\prec K\}$ are finite, Viet\'a involutions determine inductively a function 
\[
\lambda_\cM:\cH\to \cM~,
\] 
where $\lambda_\cM(H) = 3xy-z$, provided: the immediate precedents of $H$ are $H_1$ and $H_2$, $K$ is the unique common immediate precedent of $H_1$ and $H_2$, $\lambda_\cM(H_1)=x$, $\lambda_\cM(H_2)=y$, and $\lambda_\cM(K)=z$ (see Figure~\ref{pic: Markov labels}).
Because there is a Viet\'a involution that will decrease the largest entry in a Markov triple (again, aside from Markov triple $(1,1,1)$), one can check easily that $\lambda_\cM$ is surjective. 
With this terminology, MUC is merely (!) the statement that $\lambda_\cM$ is injective.

We will refer to the Markov triple $(n_1,n,n_2)$ as \emph{ordered} if the horoballs $H_1,H,H_2\in\cH$ with the respective Markov labels $n_1,n,n_2\in\cM$ appear consecutively with respect to their centers on $\bR$, i.e.~if $\lambda_\cF(H_1)<\lambda_\cF(H)<\lambda_\cF(H_2)$. For instance, $(1,5,2)$ is ordered, whereas $(2,5,1)$ is not (see Figure~\ref{pic: markov labelling}).

\subsection{Relative geometric twist}
\label{subsec: geom background}
We now define the relative geometric twist map $\tau_\X:\bQ\cap[0,1]\to\bR$.
Recall that the commutator subgroup $\Gamma<\PSL(2,\bZ)$ is a free group with generators 
\begin{equation}
\label{eq:generators}
\alpha=\begin{pmatrix} 2& 1\\1&1\end{pmatrix} \ \ \text{and}
\ \ \beta=\begin{pmatrix} 2 & -1 \\ -1 & 1\end{pmatrix}~.
\end{equation}
Note that the three systoles of $\X$ are represented by $\alpha$, $\beta$, and $\gamma=\alpha\beta$, as pictured in Figure~\ref{pic: fund domain}.
Fix the identification $H_1(\X,\bZ)\cong \bZ^2$ so that $[\alpha]=(1,-1)$ and $[\beta]=(0,1)$; note that $[\gamma]=(1,0)$. The choice for $H_1(\X,\bZ)\cong \bZ^2$ is natural given that $(q,p)\in \bZ^2 $ determines slope $\frac pq\in\bQ\cup\{\infty\}$, and the inductive description of $\lambda_\cM$ above was given initially with the label $1$ on the horoballs centered at $\infty$, $0$, and $-1$.

\begin{figure}[h]
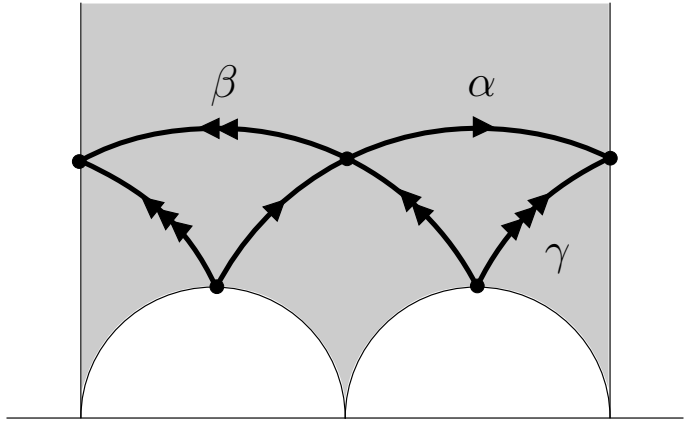

\begin{lpic}{fundDomain(9cm)}
\LARGE
\lbl[]{107,75;$\alpha$}
\lbl[]{49,75;$\beta$}
\lbl[]{124,36;$\gamma$}
\end{lpic}
\caption{Geodesic representatives for $\alpha$, $\beta$, $\gamma$ in the ideal quadrilateral with vertices $-1,0,1,\infty$, a fundamental domain for the action of $\Gamma$.}
\label{pic: fund domain}
\end{figure}

Each element $\frac pq\in\bQ\cap[0,1]$ (again, with convention that $q>0$) determines a primitive integral point $(q,p)\in  \bZ_{> 0} \times\bZ\cong H_1(\X,\bZ)$, which contains a unique oriented simple geodesic representative $\gamma_{p/q}$. 
The completion of the complement of $\gamma_{p/q}$ is a pair of pants with two geodesic boundary components, each of length $\ell_\X(\gamma_{p/q})$, and one cusp. 
The orthogonal projection of the cusp to the geodesic boundary components picks out a pair of points $A_l$ and $A_r$ on $\gamma_{p/q}$; $A_l$ is the projection of the cusp seen to the left side from $\gamma_{p/q}$, and $A_r$ is the projection of the cusp as seen to the right. 

\begin{figure}[h]
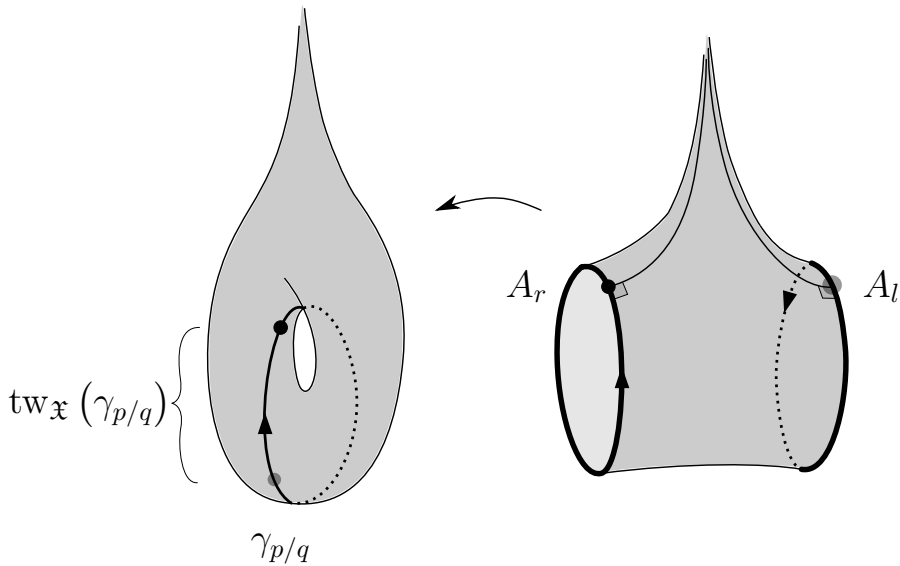

\centering
\begin{lpic}{GeomTwist(9cm)}
\Large
\lbl[]{-15,18.5;$\tw_\X\left(\gamma_{p/q}\right)$}
\lbl[]{20,-8;$\gamma_{p/q}$}
\lbl[]{65,40;$A_r$}
\lbl[]{130,40;$A_l$}
\end{lpic}
\vspace{1cm}
\caption{The relative geometric twist $\tau_\X\left(\frac pq\right)$ is the ratio of $\tw_\X(\gamma_{p/q})$ to $\ell_\X(\gamma_{p/q})$.}
\label{pic: geom twist}
\end{figure}

\begin{definition}[Geometric twist]
\label{def: geom twist}
The \emph{absolute geometric twist} $\tw_\X\left(\gamma_{p/q} \right)$ is the positive distance from $A_l$ to $A_r$ along $\gamma_{p/q}$. 
The \emph{relative geometric twist} $\tau_\X:\bQ\cap[0,1]\to[0,1]$ is given by
\[
\tau_\X\left( \frac pq \right) = \frac{ \tw_\X\left( \gamma_{p/q} \right) }{\ell_\X\left( \gamma_{p/q}\right)}~.
\]
\end{definition}

See Figure~\ref{pic: geom twist} for an illustration.

\begin{remark}
\label{rem: orientation}
There seem to be several choices made above to define $\tau_\cF$, including with the identification $H_1(\X,\bZ)\cong \bZ^2$, and with the choice of orientation for $\gamma_{p/q}$ induced by the convention that $\frac pq\in\bQ$ has $q>0$.
Note that reversing the choice of orientation of $\gamma_{p/q}$ has no affect on the absolute twist number $\tw_\X\left( \gamma_{p/q}\right)$.
In any case, it is a subjective matter whether one accepts these choices as natural; if one rejects the choice of an orientation, it might seem more natural to interpret the relative twist $\tau_\X(p/q)$ as a point of $[0,1]/\hspace{-.1cm}\sim$, where $x\sim 1-x$. 
Of course, $[0,1]/\hspace{-.15cm}\sim$ is homeomorphic to $[0,1]$, and with either this altered notion of the relative twist, or with an alternative identification $H_1(\X,\bZ)\cong \bZ^2$, the conclusion of Theorem~\ref{main thm} remains true \emph{a fortiori}.
\end{remark}

\subsection{The fractional Markov labelling}
\label{subsec: fractional Markov}
Later, we will perform explicit calculations of the geometry associated to a Markov triple (see \S\ref{sec: explicit}).
It will be helpful to introduce yet one more relevant labelling of $\cH$, the \emph{fractional Markov labelling}
$\lambda_{\cF\cM}:\cH \to \bQ\cap [0,1]$
constructed as follows: To initialize $\lambda_{\cF\cM}$, put the labels $1/1$ and $0/1$ on each of the maximal horoballs, complementary to $\cT$, centered at $\infty$ and $0$, respectively. 
Now define
\begin{equation}
\label{eq: fractional markov}
\lambda_{\cF\cM}(H) = \frac{(k_1n_1+k_2n_2)/(k_2n_1-k_1n_2)}{3n_1n_2-(k_2n_1-k_1n_2)}~,
\end{equation}
provided: the immediate precedents of $H$ are $H_1$ and $H_2$, $\lambda_{\cF\cM}(H_1)=\frac {k_1}{n_1}$, and $\lambda_{\cF\cM}(H_2)=\frac {k_2}{n_2}$. 

\begin{example}
\label{ex: kns}
We have $\lambda_{\cF\cM}(H_{1/1})=\frac 12$, $\lambda_{\cF\cM}(H_{1/2})=\frac 25$, $\lambda_{\cF\cM}(H_{1/3})=\frac 5{13}$, and $\lambda_{\cF\cM}(H_{2/3})=\frac{12}{29}$.
\end{example}

\begin{remark}
\label{rem: lambda length}
The reader may note that the inductive description of $\lambda_{\cF\cM}$ seems to rely only on the labels $\lambda_{\cF\cM}(H_1)$ and $\lambda_{\cF\cM}(H_2)$ of the immediate precedents, and not on the label $\lambda_{\cF\cM}(K)$ of the common immediate precedent $K$ of $H_1$ and $H_2$, in contrast to the Markov labelling $\lambda_\cM$.
This is because the data of $\lambda_\cM(K)$ is already contained in \eqref{eq: fractional markov}. 
In fact, the quantity $k_2n_1-k_1n_2$ appearing in \eqref{eq: fractional markov} is equal to the \emph{$\lambda$-length} of the geodesic arc in $\bH^2$ between $\frac{k_1}{n_1}$ and $\frac{k_2}{n_2}$ -- for our purposes the reader may take the definition of the $\lambda$-length of the geodesic between $\frac ab$ and $\frac pq$ to be $|pb-aq|$.
It will be a byproduct of the proof of Proposition~\ref{prop: fractional markov} that $k_2n_1-k_1n_2$ is in fact equal to $\lambda_{\cM}(K)$. 
For more about $\lambda$-length the reader can consult \cite{Fomin-Thurston,Penner}.
\end{remark}

\begin{remark}
\label{rem: elliptic}
We will make repeated use of the fact that every complete hyperbolic surface homeomorphic to a punctured torus has an \emph{elliptic isometry}, an order-two orientation-preserving isometry which fixes three points $w_1,w_2,w_3$, the \emph{Weierstrass points} of the surface 
(see \cite{Farb-Margalit}).
\end{remark}

Though both the numerator and denominator of $\lambda_{\cF\cM}$ may seem strange, they both have important geometric meaning. (See also \cite{Springborn2} for related work.)

\begin{proposition}
\label{prop: fractional markov}
Suppose that $H\in \cH$ and $\lambda_{\cF\cM}(H)=\frac kn$ with $\gcd(k,n)=1$. 
\begin{enumerate}[label = (\roman*)]
\item $\lambda_\cM(H)=n$ 
\item $k^2\equiv-1\mod n$
\item The projection of $\{\Re(z)=\frac kn\}$ to $\X$ is the simple proper geodesic arc disjoint from $\gamma_{\lambda_\cF(H)}$, \\the simple closed geodesic whose homology class is parallel to the line of slope $\lambda_\cF(H)$.
\item Suppose that $(H_1,H_2)$ are the immediate precedents of $H$, and that $\lambda_{\cF\cM}(H_i)=\frac{k_i}{n_i}$.
Then $\frac{k_1}{n_1}<\frac kn<\frac{k_2}{n_2}$, and the projections of $\{\Re(z)=\frac kn\}$, $\{\Re(z)=\frac {k_1}{n_1}\}$, and $\{\Re(z)=\frac {k_2}{n_2}\}$ form an ideal triangulation of $\X$, in which the counterclockwise order of the arcs on $\X$ is given by the projections of $\{\Re(z)=\frac {k_1}{n_1}\}$, $\{\Re(z)=\frac{k_2}{n_2}\}$, and $\{\Re(z)=\frac k n\}$, in that order.
\end{enumerate}
\end{proposition}

\begin{remark}
\label{rem: prime powers}
It is immediate from this proposition that MUC holds for Markov numbers $n=p^r$ which are powers of primes, since it is well known that such numbers have at most one root of $-1$ in $\{1,\ldots,\frac{n-1}2\}$. This corollary is well known, and now has several proofs \cite{Aigner,Baragar2,Button,Lang-Tan,Schmutz,Zhang}.
\end{remark}

\begin{proof}
Observe first that each of the first two properties in fact follow from the third.
Indeed, the first calculation is built on the well-known observation that the $\lambda$-length of the simple proper geodesic arc disjoint from $\gamma_{\lambda_\cF(H)}$ is given by $\lambda_\cM(H)$, and the second is built on the observation that every such arc is preserved by the elliptic involution. 
We leave the details to the reader.

We prove the last two points simultaneously, using the inductive description of $\lambda_{\cF\cM}$. At the root of the tree, observe that one sees the labels $-1/1$, $1/1$, and $0/1$ on the horoballs centered at $-1$, $\infty$, and $0$, respectively. 
The reader may use Figure~\ref{pic: fund domain} to check that $\Re(z)=-1$, $\Re(z)=1$, and $\Re(z)=0$ project to geodesic arcs on $\X$ that are disjoint from $\alpha$, $\beta$, and $\gamma$, respectively, and so that the arcs are seen in the counterclockwise order $\Re(z)=-1$, $\Re(z)=1$, $\Re(z)=0$ on $\X$.

For the inductive step, suppose that $(H_1,H_2)$ are the immediate precedents of $H$, and that $K$ is the common immediate precedent of $H_1$ and $H_2$, as in Figure~\ref{pic: Markov labels}.
Let $\lambda_\cF(H_1)=\frac{p_1}{q_1}$, $\lambda_\cF(H_2)=\frac{p_2}{q_2}$, $\lambda_{\cF\cM}(H_1)=\frac{k_1}{n_1}$, $\lambda_{\cF\cM}(H_2)=\frac{k_2}{n_2}$, and $\lambda_{\cF\cM}(K)=\frac{k_0}{n_0}$.
By inductive hypothesis, the arcs $\Re(z)=\frac{k_i}{n_i}$ project to arcs $\alpha_i$ forming an ideal triangulation of $\X$, where $\alpha_i$ is disjoint from the corresponding simple geodesic $\gamma_{p_i/q_i}$, and $\frac{k_1}{n_1}<\frac{k_2}{n_2}$.
Moreover, we observe that $n_0\le \min\{n_1,n_2\}$, so by inductive hypothesis the arcs appear in the counterclockwise order $\alpha_1$, $\alpha_0$, $\alpha_2$ on $\X$. (Note that, since we do not know which is bigger among $n_1$ and $n_2$, we don't know whether $\frac{k_0}{n_0} < \frac{k_1}{n_1} < \frac{k_2}{n_2}$ or $\frac{k_1}{n_1}<\frac{k_2}{n_2} < \frac{k_0}{n_0}$.)

It is straightforward that there are precisely two other simple geodesic arcs on $\X$, intersecting once, that form an ideal triangulation with $\alpha_1$ and $\alpha_2$.
One of these is evidently $\alpha_0$; let $\alpha$ be the other.
We will show shortly that $\alpha$ has a lift given by $\Re(z) = \lambda_{\cF\cM}(H)$, between $\lambda_{\cF\cM}(H_1)$ and $\lambda_{\cF\cM}(H_2)$, and that the arcs appear in counterclockwise order $\alpha_1$, $\alpha_2$, $\alpha$ on $\X$. Note that this will complete the proof, as it is also apparent that $\alpha$ is disjoint from the curve homologous to $[\gamma_{p_1/q_1}]+[\gamma_{p_2/q_2}]$, i.e. to $[\gamma_{\lambda_\cF(H)}]$.

Now the two arcs $\alpha_0$ and $\alpha$ admit lifts inside an ideal quadrilateral $Q$ bounded by $\Re(z)=\frac{k_1}{n_1}$ and $\Re(z)=\frac{k_2}{n_2}$. One of these will be the geodesic in $\bH^2$ with endpoints $\frac{k_1}{n_1}$ and $\frac{k_2}{n_2}$, and the other will be a vertical geodesic $\Re(z)=\frac kn$ between $\frac{k_1}{n_1}$ and $\frac{k_2}{n_2}$. 
Because the arcs on $\X$ appear in counterclockwise order $\alpha_1$, $\alpha_0$, $\alpha_2$ around a complementary ideal triangle, it must be that the lift of $\alpha$ inside $Q$ is the vertical geodesic $\Re(z)=\frac kn$. See Figure~\ref{pic: quadAlphas}.
(Note that the $\lambda$-length of $\alpha_0$ is $k_2n_1-k_1n_2$; hence $n_0=k_2n_1-k_1n_2$.)

\begin{figure}
\begin{minipage}{.45\textwidth}
\centering
\begin{lpic}{quadAlphas(,5cm)}
\Large
\lbl[]{0,60;$\widetilde{\alpha}_1$}
\lbl[]{85,60;$\widetilde{\alpha}_2$}
\lbl[]{50,70;$\widetilde{\alpha}$}
\lbl[]{60,40;$\widetilde{\alpha}_0$}
\end{lpic}
\vspace{.5cm}
\caption{The vertical lifts of $\alpha_1$ and $\alpha_2$ form an ideal triangulation with two other arcs, one of which is vertical.}
\label{pic: quadAlphas}
\end{minipage}\hfill
\begin{minipage}{.45\textwidth}
\centering
\begin{lpic}{kns(,5cm)}
\Large
\lbl[]{12,-10;$\frac{k_1}{n_1}$}
\lbl[]{35,-10;$\frac kn$}
\lbl[]{70,-10;$\frac{k_2}{n_2}$}
\lbl[]{135,-10;$3+\frac{k_1}{n_1}$}
\end{lpic}
\vspace{.5cm}
\caption{Lifts of the arcs $\alpha_1$, $\alpha$, $\alpha_2$ in consecutive order around the puncture.}
\label{pic: kns} 
\end{minipage}
\end{figure}

The map $z\mapsto z+3$ projects to the elliptic involution of $\X$. 
Starting at $i$, traversing the maximal horoball centered at $\infty$ positively, we see that the lifts of the arcs of the triangulation have real parts $\frac{k_1}{n_1}$, $\frac kn$, $\frac{k_2}{n_2}$, $3+\frac{k_1}{n_1}$, $3+\frac kn$, $3+\frac{k_2}{n_2}$ (see Figure~\ref{pic: kns}). There is another lift of the elliptic involution that preserves the vertical arc $\Re(z)=\frac{k_2}{n_2}$, given by
\[
\begin{pmatrix}
k_2 & -\frac{1+k_2^2}{n_2} \\
n_2 & -k_2
\end{pmatrix}~.
\]
This involution takes $\Re(z)=3+\frac{k_1}{n_1}$ (a geodesic between $\infty$ and $3+\frac{k_1}{n_1}$) to a geodesic between $\frac{k_2}{n_2}$ and the finite endpoint of the geodesic $\Re(z)=\frac kn$, i.e.~to $\frac kn$. Therefore we find that
\begin{align*}
\frac kn & = \begin{pmatrix}
k_2 & -\frac{1+k_2^2}{n_2} \\
n_2 & -k_2
\end{pmatrix} \cdot \left( 3 + \frac{k_1}{n_1} \right) 
= \frac { 3k_2n_1+ k_1k_2 - \frac{1+k_2^2}{n_2} n_1}{3n_1n_2-(k_2n_1-k_1n_2)}~.
\end{align*}
Using the fact that $k_2n_1-k_1n_2=n_0$, and that $(n_0,n_1,n_2)$ is a Markov triple, the reader may check  
\[
3k_2n_1+ k_1k_2 - \frac{1+k_2^2}{n_2} n_1 = \frac{ k_1n_1+k_2n_2}{k_2n_1-k_1n_2}~,
\]
so we are done.
\end{proof}

\subsection{Comparing the labellings}
Towards Theorem~\ref{main thm}, our strategy is to estimate $\tau_\X$ by $\tau_\cF$. 
The labellings $\lambda_\cF$ and $\lambda_\cM$ are the key to passing back and forth between the hyperbolic geometry of $\X$ and the arithmetic of $\cF$. 
Note that this strategy for studying the Markov numbers is classical. It is the key tool, for instance, in Zagier's approach to counting the number of Markov numbers below a given bound \cite{Zagier} (though Zagier replaces the Farey tree $\cT_0$ with the `Euler tree', in which one only remembers the denominators of the Farey fractions).

The relationship between Farey triples and Markov triples is captured as follows:

\begin{proposition}
\label{prop: topology of triples}
Suppose that $\frac ab< \frac pq < \frac cd$ is a Farey triple. 
Then $\lambda_\cM\circ\lambda_\cF^{-1}\left( \frac ab, \frac pq, \frac cd\right)$ is an ordered Markov triple, and the
geodesics $\gamma_{a/b}$, $\gamma_{p/q}$, $\gamma_{c/d}$ form a trio of simple closed geodesics that pairwise intersect once, whose complementary components contain a pair of triangles exchanged by the elliptic isometry.
Moreover, the given geodesics appear in the counterclockwise order $\gamma_{c/d}$, $\gamma_{p/q}$, $\gamma_{a/d}$ around the boundary of either triangle on $\X$.
\end{proposition}

\begin{figure}
\begin{minipage}{.3\textwidth}
\centering
\captionsetup{width=.8\linewidth}
\begin{lpic}{labelling(4.5cm)}
\Large
\lbl[]{13,42;$\frac 01$}
\lbl[]{57,42;$\frac 11$}
\lbl[]{37,22;$\frac 12$}
\lbl[]{10,12;$\frac 13$}
\lbl[]{63,12;$\frac 23$}
\lbl[]{22,-2;$\frac 25$}
\lbl[]{52,-2;$\frac 35$}
\end{lpic}
\vspace{.5cm}
\caption{The Farey labelling $\lambda_\cF$.}
\label{pic: farey labelling}
\end{minipage}\hfill
\begin{minipage}{.34\textwidth}
\centering
\captionsetup{width=.8\linewidth}
\begin{lpic}{labelling(4.5cm)}
\large
\lbl[]{13,42;$1$}
\lbl[]{57,42;$2$}
\lbl[]{37,22;$5$}
\lbl[]{10,13;$13$}
\lbl[]{63,13;$29$}
\lbl[]{21,-2;$194$}
\lbl[]{53,-2;$433$}
\end{lpic}
\vspace{.5cm}
\caption{The Markov labelling $\lambda_\cM$.}
\label{pic: markov labelling}
\end{minipage}\hfill
\begin{minipage}{.34\textwidth}
\centering
\captionsetup{width=.8\linewidth}
\begin{lpic}{labelling(4.5cm)}
\large
\lbl[]{13,42;$\frac 01$}
\lbl[]{57,42;$\frac 12$}
\lbl[]{37,22;$\frac 25$}
\lbl[]{10,13;$\frac 5{13}$}
\lbl[]{63,13;$\frac{12}{29}$}
\lbl[]{21,-2;$\frac{75}{194}$}
\lbl[]{53,-2;$\frac{179}{433}$}
\end{lpic}
\vspace{.5cm}
\caption{The fractional Markov labelling $\lambda_{\cF\cM}$.}
\label{pic: fractional markov labelling}
\end{minipage}
\end{figure}

\begin{remark}
\label{rem: orientation switch}
The reader should notice the switch in orientation from the cyclic order $\frac ab, \frac pq, \frac cd$ to the cyclic order $\gamma_{c/d}$, $\gamma_{p/q}$, $\gamma_{a/b}$ around the boundary of the complementary triangles. 
We emphasize: 
the ordered Markov triple $(n_1,n,n_2)$ gives rise to oriented simple geodesics $\gamma_1$, $\gamma$, $\gamma_2$ so that each complementary triangle sees the curves in the (counterclockwise) cyclic order $\gamma_2$, $\gamma$, $\gamma_1$.
\end{remark}

The proof of this proposition amounts to bookkeeping: First, observe that the geometric intersection numbers among the three curves are equal to one by assumption. Because each simple geodesic on $\X$ passes through two of the three Weierstrass points, the intersection points must be the three Weierstrass points. One can check that a trio of simple curves  on a punctured torus, pairwise intersecting in three distinct points, must have complementary regions which are a pair of triangles, exchanged by the elliptic involution, and one hexagonal region containing the cusp.

As for the cyclic order on the triangular sides, observe that the initial labelling used to construct $\lambda_\cF$ and $\lambda_\cM$ started with the triples $ \frac 01, \frac 10, \frac 1{-1}$ and corresponding geodesics represented by $[\alpha]=\gamma_{0/1}$, $[\beta\alpha]=\gamma_{1/0}$, and $[\beta]=\gamma_{1/-1}$, as pictured in Figure~\ref{pic: fund domain}. 
As can be observed in Figure~\ref{pic: fund domain}, either complementary triangle to this trio produces the cyclic order $\gamma_{0/1}$, $\gamma_{1/-1}$, $\gamma_{1/0}$, evidently the reverse of the cyclic order $ \frac 01, \frac 10, \frac 1{-1}$ in the Farey graph.
To fill in the labels for either $\lambda_\cF$ or $\lambda_\cM$ on $\cH$, one applies the inductive procedure along $\cT$. In either case, this procedure is obtained by applying a Dehn twist (either in the guise of a Viet\'a involution, which is the map induced on traces by a Dehn twist, or by a transvection on $\bQ$, the action of a Dehn twist in $H_1(\X,\bZ)$), which preserves the order reversal observed at the base.

It also deserves mentioning that there is a certain consistency between the labellings $\lambda_\cF$ and $\lambda_\cM$. Namely, if $(H_1,H_2)$ are the immediate precedents of $H$, then one can check by induction that $\lambda_\cF(H_1) <\lambda_\cF(H)<\lambda_\cF(H_2)$ and $\lambda_{\cF\cM}(H_1) <\lambda_{\cF\cM}(H)<\lambda_{\cF\cM}(H_2)$. Hence,

\begin{proposition}
\label{prop: orders}
For all $H,K\in \cH$, $\lambda_\cF(H) < \lambda_\cF(K)$ if and only if $\lambda_{\cF\cM}(H) < \lambda_{\cF\cM}(K)$.
\end{proposition}

Note that $\lambda_\cF$ is evidently a bijection, so by the Proposition $\lambda_{\cF\cM}$ is injective as well. This should be contrasted with the deep mystery of MUC, which predicts that $\lambda_\cM$ is injective.

\bigskip
\section{The graph of the Farey twist function}
\label{sec: Farey analysis}
Towards Theorem~\ref{thm: prescribe arith twist}, we seek to understand the graph of $\tau_\cF:\bQ\cap[0,1]\to\bQ\cap[0,1]$. 
Observe that, if $\frac pq$ has immediate precedents $\left( \frac ab, \frac cd\right)$, then the immediate successors of $\frac pq$ are $\left( \frac {a+p}{b+q},\frac{p+c}{q+d}\right)$, whose immediate precedents are in turn $\left( \frac ab, \frac pq\right)$ and $\left( \frac pq, \frac cd\right)$, respectively.

We find
\[
\tau_\cF \left( \frac pq \right) = \frac b q~, \ \ 
\tau_\cF \left( \frac{a+p}{b+q} \right) = \frac b{b+q}~, \ \ \text{and}  \ \ 
\tau_\cF \left( \frac{p+c}{q+d} \right) = \frac q{q+d}~.
\]
Letting $z=\tau_\cF(p/q)$, it is immediate to verify that
\[
\tau_\cF \left( \frac{a+p}{b+q} \right) = \frac z{z+1}
\ \ \text{and} \ \ 
\tau_\cF \left( \frac{p+c}{q+d} \right) = \frac 1{2-z}~.
\]
Let us therefore define 
$L(z)=\frac z{z+1}$ and $R(z)=\frac1{2-z}$.
The above calculation tells us how the Farey twist numbers change as we move down the Farey tree: the Farey twist number $z$ gives rise to Farey twist number $L(z)$ immediately to the left, and Farey twist number $R(z)$ immediately to the right.

Let $\Lambda$ be the semigroup of linear fractional transformations $[0,1]\to[0,1]$ generated by $L$ and $R$. Note that $L([0,1])=[0,\frac12]$ and $R([0,1])=[\frac12,1]$. Moreover, both $L$ and $R$ are \emph{distance-decreasing} on $[0,1]$, in that $|L(I)|<|I|$ and $|R(I)|<|I|$ for every interval $I\subset [0,1]$: both $L'(z) = \frac1{(z+1)^2}$ and $R'(z) = \frac1{(2-z)^2}$ are strictly less than 1 on $(0,1)$, so
\[
|f(I)| = \int_I |f'(z)| \; dz < |I|
\]
for $f$ equal to either $L$ or $R$.

\begin{proposition}
\label{prop: dense orbits}
For every $x\in[0,1]$, the orbit $\Lambda\cdot x\subset [0,1]$ is dense.
\end{proposition}

\begin{proof}
Let $J=\overline{ \Lambda \cdot x}$. 
Observe that $L^n\cdot x$ and $R^n\cdot x$ approach $0$ and $1$. 
Hence $0,1\in J$, and $\frac12=R(0)\in J$.

Towards contradiction, suppose that $U$ is an open interval in the complement $[0,1]\setminus J$ which is of maximal length, in that $|U|\ge|V|$ for every open interval $V\subset [0,1]\setminus J$. Because $U\subset [0,1]\setminus J$ and $\frac12\in J$, we have $\frac12\notin U$, and thus $U$ is contained in either $L([0,1])$ or $R([0,1])$. In either case, we find an open interval in $[0,1]\setminus J$ (either $L^{-1}U$ or $R^{-1}U$) of size strictly larger than $|U|$, a contradiction. 
\end{proof}

\begin{remark}
One can use (a semigroup variant of) the Ping-Pong Lemma to show that $\Lambda$ is freely generated by $L$ and $R$: in the language of \cite[Lem.~2.2]{Eskin-Mozes-Oh}, $L$ and $R$ can be recovered as $A^2\circ B$ and $A\circ B$, for $A(z)=\frac{z-1}{3z-2}$ and $B(z)=\frac{z-1}{2z-1}$, where $B([0,1])\cap [0,1]=\emptyset$.
\end{remark}

Now we prove Theorem~\ref{thm: prescribe arith twist}. Consider $(x,y)\in[0,1]^2$, and let $\epsilon>0$. 
It is easy to see that there exists a Farey triple $\frac ab<\frac pq < \frac cd$ contained in $(x-\epsilon,x+\epsilon)$: 
choose an interval $(w-\delta,w+\delta)\subset(x-\epsilon,x+\epsilon)$ for $\delta>0$, for $w$ a quadratic irrational, and an element $A\in\PSL(2,\bZ)$ with attracting fixed point $w$. 
The repelling fixed point of $A$ is the Galois conjugate of $w$, evidently not $0$, $1$, or $\infty$. 
Hence $A^nz\to w$ for $z\in\{0,1,\infty\}$, so for $n$ large enough $A^n\cdot\{0,1,\infty\}$ are the vertices of a Farey triple in $(w-\delta,w+\delta)\subset (x-\epsilon,x+\epsilon)$.

Now Proposition~\ref{prop: dense orbits} implies that there is an element $g\in \Lambda=\langle L,R\rangle$ so that $g\cdot \tau_\cF\left(\frac pq\right) \in (y-\epsilon,y+\epsilon)$. 
Let $\frac{p'}{q'}$ be the vertex of $\cF$ obtained by following the left-right turns downward along $\cT$, as dictated by $g$. 
It follows that $\tau_\cF\left(\frac{p'}{q'}\right) = g\cdot \tau_\cF\left(\frac pq\right)$.
By construction, the vertex $\frac{p'}{q'}$ is evidently a $\prec$-successor of $\frac pq$.
Therefore Lemma~\ref{lem: orders} implies that $\frac ab< \frac{p'}{q'}<\frac cd$, so that $\frac{p'}{q'}\in(x-\epsilon,x+\epsilon)$. Hence
\[
\left| \left( \frac{p'}{q'} , \tau_\cF\left( \frac{p'}{q'}\right) \right) - (x,y) \right| < 2\epsilon ~,
\]
as desired.

\bigskip
\section{Computing the absolute geometric twist}
\label{sec:compute geom}
Here we examine in more detail the absolute twist number of an oriented simple geodesic on a finite-volume hyperbolic surface $\cY$ homeomorphic to $\X$, using hyperbolic geometry to obtain a description that is more amenable to comparison with the Farey twist. 
The reader should view $\cY$ as fixed throughout this section; hence we will suppress the burdensome subscripts `$\cY$' (e.g.~$\tw(\gamma):=\tw_\cY(\gamma)$ and $\tau(\gamma):=\tau_\cY(\gamma)$).
Below, we make repeated use of the elliptic involution of $\cY$. 
Note that in \S\ref{sec: other structures}, we will observe that more or less the same computations hold when $\cY$ has geodesic boundary.

Suppose that $(\gamma_2,\gamma,\gamma_1)$ is a trio of simple closed geodesics pairwise intersecting once (such as the trio arising from a Markov triple), whose complement contains a pair of triangles that are exchanged by the elliptic involution of $\cY$, each of which sees the curves in the given cyclic order. 
The vertices of both triangles are $w_1,w_2,w_3$, and each curve passes through two Weierstrass points; suppose that $\gamma$ passes through $w_1$ and $w_2$, evidently diametrically opposed points on $\gamma$, and let $T$ be the triangle with vertices $w_1$, $w_2$, $w_3$ in counterclockwise order.
With this geometric picture in mind, it will be convenient to record the half-lengths of the curves: let $\ell:=\frac12\ell(\gamma)$ and $\ell_i:=\frac12\ell(\gamma_i)$.\\

\begin{figure}[h]
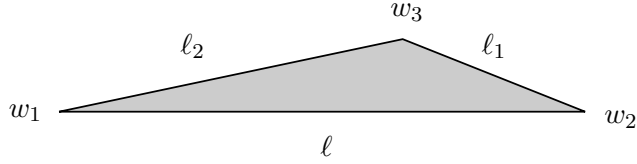

\vspace{.2cm}
\begin{lpic}{triangle(7cm)}
\lbl[]{20,10;$\ell_2$}
\lbl[]{40,-5;$\ell$}
\lbl[]{65,10;$\ell_1$}
\lbl[]{-5,0;$w_1$}
\lbl[]{84,-1;$w_2$}
\lbl[]{52,15;$w_3$}
\end{lpic}
\vspace{.5cm}
\caption{The triangle $T$, complementary to the triple $\gamma_2$, $\gamma$, $\gamma_1$.}
\label{pic: triangle}
\end{figure}

Choose an oriented lift $\tilde{\gamma}$ of $\gamma$ in the cover $\pi:\bH^2\to\cY$ (see Remark~\ref{rem: orientation}), and suppose that $g$ is a generator from the covering group of the cyclic subgroup preserving $\tilde{\gamma}$, preserving the orientation of $\tilde{\gamma}$.
The geodesic completion of $\cY\setminus \gamma$ is a hyperbolic pair of pants with one cusp, and it follows that there are precisely two simple orthogeodesics in $\cY$ from $\gamma$ to the cusp. 
The lifts of these orthogeodesics that are incident to $\tilde{\gamma}$ intersect $\tilde{\gamma}$ at a sequence of points preserved by $g$, 
which comprise two $\langle g\rangle$-orbits,
according to whether the corresponding orthogeodesic lifts lie to the left or right of $\tilde{\gamma}$; we call these lifts \emph{left} and \emph{right} $\tilde{\gamma}$-{orthogeodesics}, and the intersection points with $\tilde{\gamma}$ the \emph{left} and \emph{right foot points}. (Note that we allow the possibility that the two $\langle g\rangle$-orbits are identical, in which case each point in the orbit is both a left and right foot point.)
Evidently, every pair of distinct left (resp.~right) foot points is at distance $\ge 2\ell$, and the left and right foot points alternate along $\tilde{\gamma}$.

Choose a left foot point $A_l$, and let $A_r$ be the right foot point at non-negative distance from $A_l$ along $\tilde{\gamma}$ (that is, again $A_r=A_l$ is possible). 
Recall that the absolute geometric twist $\tw(\gamma)$ is the non-negative distance along the oriented geodesic $\tilde{\gamma}$ from the point $A_l$ to the point $A_r$.
Because $\gamma$ passes through $w_1$ and $w_2$ on $\X$, there are a pair of Weierstrass points $\widetilde{w_1}$ and $\widetilde{w_2}$ on $\tilde{\gamma}$ between $A_l$ and $g\cdot A_l$;
suppose that $\widetilde{w_1}$ is the first lift of a Weierstrass point on $\tilde{\gamma}$ at non-negative distance from $A_l$ (again, $\widetilde{w_1}=A_l$ is \emph{a priori} possible), and let $\widetilde{w_2}$ be the following one. (Strictly speaking, we do not yet know that $\widetilde{w_1}\in\pi^{-1}(w_1)$, since $\widetilde{w_1}\in\pi^{-1}(w_2)$ may seem possible -- this will be ruled out shortly, but suffice it to say $\pi(\widetilde{w_1})=w_1$ holds because of the assumption on the cyclic order in Figure~\ref{pic: triangle}.)
Let $\alpha_l$ and $\alpha_r$ be the left and right $\tilde{\gamma}$-orthogeodesics incident to $A_l$ and $A_r$, respectively.

\begin{lemma}
\label{lem: absolute twist}
The point $\widetilde{w_1}$ lies between $A_l$ and $A_r$ on $\tilde{\gamma}$, and $\tw(\gamma) = 2 d(A_l, \widetilde{w_1})$.
\end{lemma}

\begin{proof}
Let $\phi$ be the lift of the elliptic isometry of $\cY$ that fixes $\widetilde{w_1}$, and observe that $\phi$ interchanges left and right $\tilde{\gamma}$-orthogeodesics.

We claim that $\phi \cdot \alpha_l=\alpha_r$. 
Indeed, $d(A_l,\widetilde{w_1})< \ell$ -- note that if it were equal to $\ell$, then in that case $A_l$ would be a lift of a Weierstrass point on $\tilde{\gamma}$, so in fact $\widetilde{w_1}$ would be equal to $A_l$ by construction -- so $\phi \cdot \alpha_l$ is a right $\tilde{\gamma}$-orthogeodesic that intersects $\tilde{\gamma}$ at $\phi \cdot A_l$, at distance $<2\ell$ from $A_l$. Hence $\phi \cdot \alpha_l$ and $\alpha_r$ are both right $\tilde{\gamma}$-orthogeodesics strictly between $A_l$ and $g\cdot A_l$, so $\phi \cdot \alpha_l=\alpha_r$.

It follows that $\phi \cdot A_l = A_r$, and so $\widetilde{w_1}$ is the midpoint of $A_l$ and $A_r$. The result follows, as $\tw(\gamma) = d(A_l,A_r) = 2d(A_l, \widetilde{w_1})$.
\end{proof}

As for the third Weierstrass point $w_3$, 
let $\eta$ be the geodesic between $\alpha_l(\infty)$ and $g\cdot \alpha_l(\infty)$, the endpoints of $\alpha_l$ and $g\cdot \alpha_l$ respectively.
One can check that the projection of $\eta$ under the covering map $\bH^2\to\cY$ is an embedding (since it is easy to construct a representative arc which is simple), and the image is a simple complete geodesic arc disjoint from $\gamma$. 
It follows that the point $\widetilde{w_3}$ of $\eta$ closest to $\tilde{\gamma}$ must satisfy $\pi(\widetilde{w_3})=w_3\in\cY$: The arc $\eta$ projects to a simple geodesic from the cusp to itself on the pair of pants $\cY\setminus \gamma$. The pair of pants has an evident order-two orientation-preserving isometry that exchanges geodesic boundary components, with one fixed point at the projection of $\widetilde{w_3}$. It follows that this isometry induces the elliptic isometry of $\cY$, and so the projection of $\widetilde{w_3}$ is a Weierstrass point. This Weierstrass point is on the projection of $\eta$, so it is not on $\gamma$, and hence $\widetilde{w_3}\in\pi^{-1}(w_3)$.

Now we may observe that $\pi(\widetilde{w_i})=w_i$ as expected: if instead $\pi(\widetilde{w_1})=w_2$, then the left orthogeodesic $\alpha_l$ leaves $\tilde{\gamma}$ interior to a lift of $T$. In that case, one of the two sides of lengths $\ell_1$ and $\ell_2$ from $\widetilde{w_3}$ would be longer than $\ell$, a contradiction.

We summarize:

\begin{proposition}
\label{prop: picture of lift}
There is a quadrilateral $Q$ in $\bH^2$, with two consecutive right-angles at $A_l$ and $g\cdot A_l$, and two ideal vertices $\alpha_l(\infty)$ and $g\cdot\alpha_l(\infty)$. Moreover, $\partial Q$ contains points $\widetilde{w_i}\in\pi^{-1}(w_i)$ so that $\widetilde{w_1},\widetilde{w_2}$ are consecutive from $A_l$ to $g\cdot A_l$, with $d(A_l,\widetilde{w_1})=\frac12\tw(\gamma)$, and where $\widetilde{w_3}$ is the point of the side between $\alpha_l(\infty)$ and $ g\cdot\alpha_l(\infty)$ projecting to the midpoint between $A_l$ and $g\cdot A_l$. 
The triangle $\widetilde{w_1}\widetilde{w_2}\widetilde{w_3}$ has side lengths $\ell =d(\widetilde{w_1},\widetilde{w_2})$, $\ell_1= d(\widetilde{w_2},\widetilde{w_3})$, and $\ell_2=d(\widetilde{w_1},\widetilde{w_3})$.
\end{proposition}

\begin{figure}
\begin{lpic}{SaccheriQuad(9cm)}
\lbl[]{15,-5;$A_l$}
\lbl[]{180,-5;$g\cdot A_l$}
\lbl[]{48,7;$x$}
\lbl[]{69,-6;$\widetilde{w_1}$}
\lbl[]{139,-6;$\widetilde{w_2}$}
\lbl[]{96,-5;$M$}
\lbl[]{97,36;$\widetilde{w_3}$}
\lbl[]{76,17;$\ell_2$}
\lbl[]{125,17;$\ell_1$}
\end{lpic}
\vspace{.8cm}
\caption{The twist $\tw \left(\gamma \right)$ is $2x=2d(M,\widetilde{w_2})$.}
\label{pic: quadrilateral}
\end{figure}

Now let $x=\frac12 \tw(\gamma)=d(A_l,\widetilde{w_1})$, and let $M\in \partial Q$ be the midpoint of $A_l$ and $g\cdot A_l$, evidently a point between $\widetilde{w_1}$ and $\widetilde{w_2}$. Because $d(A_l,\widetilde{w_1})=x$ and $d(\widetilde{w_1},\widetilde{w_2})=\ell$, we find that $d(M,\widetilde{w_2})=x$ as well.

Hyperbolic trigonometry in the quadrilateral with vertices $A_l$, $M$, $\widetilde{w_3}$, and $\alpha_l(\infty)$ (that is, half of $Q$) can be used \cite[\S8.5]{Fenchel-Nielsen} to demonstrate that 
\begin{equation}
\label{eq: ideal lambert}
\sinh \ell \cdot \sinh d(M,\widetilde{w_3}) =1~.
\end{equation}
Thus $d(M,\widetilde{w_3}) =f(\ell)$, where $f(x) = \sinh^{-1}\left( \frac1{\sinh x} \right)$.

Hence $M\widetilde{w_2}\widetilde{w_3}$ is a right-angled hyperbolic triangle with hypotenuse length $\ell_1$ and side lengths $x$ and $f(\ell)$.
By the Pythagorean Theorem for hyperbolic triangles,
\begin{equation}
\label{eq: pythagoras}
\cosh \ell_1 = \cosh x \cosh f(\ell) = \cosh x \coth \ell~.
\end{equation}
Rearranging, we find that
\begin{equation}
\label{eq:x}
x= \cosh^{-1}\left( \tanh \ell \cosh \ell_1 \right) =\ell_1 + O\left( e^{-2\ell}\right)~.
\end{equation}

We have proved:

\begin{proposition}
\label{prop: estimate geom twist}
There is a constant $C=C(\cY)>0$ so that $\left| \tw_\cY(\gamma)- \ell_\cY(\gamma_1)\right|\le Ce^{-\ell_\cY(\gamma)}$.
\end{proposition}

\bigskip
\section{Comparing arithmetic and geometric twist numbers}
\label{sec: comparison}

Recall that $\tau_\X,\tau_\cF:\bQ\cap[0,1]\to[0,1]$ are the relative geometric and arithmetic twist functions, respectively.
In this section, we prove Proposition~\ref{prop:comparing twists}, i.e.~the inequality
\begin{equation}
\label{eq: control}
\left | \tau_\X\left( \frac pq \right) - \tau_\cF \left( \frac pq\right) \right| \le \frac C{q^2}
\end{equation}
for a constant $C=C(\X)>0$. 

Suppose that $\frac pq$ is the center of the Farey triple $\frac {p_1}{q_1} < \frac pq < \frac {p_2}{q_2}$. 
Let $\gamma_i=\gamma_{p_i/q_i}$ and $\gamma=\gamma_{p/q}$. 
Keeping notation as in \S\ref{sec:compute geom}, we find that
\[
\tau_\cF\left( \frac pq\right) = \frac{q_1}q \ , \ \ \text{and} \ \ \tau_\X\left(\frac pq\right) = \frac {\frac12\tw_\X(\gamma)}{\ell}~.
\]

By Proposition~\ref{prop: estimate geom twist}, we find that
\[
\left| \tau_\X\left(\frac pq\right) - \frac{\ell_1}\ell \right| \le Ce^{-2\ell}~.
\]

Recall Fock's function $\Psi:[0,1]\to\bR$, a continuous, positive function, satisfying $\Psi\left(\frac pq\right) = \frac 1q \ell_\X\left( \gamma_{p/q}\right)$. It follows that
\[
\frac{\ell_1}\ell = \frac{q_1 \; \Psi\left(\frac{p_1}{q_1}\right)}{q \; \Psi\left( \frac pq \right)}~,
\]
so that
\begin{align*}
\left |\; \tau_\X\left( \frac pq \right) - \tau_\cF \left( \frac pq\right) \;\right| \le 
\left|\; \frac{q_1}q - \frac{\ell_1}\ell\; \right| + Ce^{-2\ell} = \frac{q_1}{q\;\Psi\left(\frac pq\right)} \left| \;\Psi\left(\frac pq\right) - \Psi\left(\frac{p_1}{q_1}\right) \; \right| + Ce^{-2\ell}~.
\end{align*}
Because $\Psi$ is convex, it is Lipschitz. Because it is bounded away from zero, $e^{-2\ell}=e^{-2q\Psi(p/q)}<e^{-k q}$ for some constant $k>0$. Rechoosing $C$ if necessary,
\[
\left |\; \tau_\X\left( \frac pq \right) - \tau_\cF \left( \frac pq\right) \;\right| 
\le C\; \frac{q_1}{q}  \left| \frac pq - \frac{p_1}{q_1} \right| + Ce^{-kq}
= \frac C{q^2} + Ce^{-kq}~,
\]
and \eqref{eq: control} follows.

\bigskip
\section{Other hyperbolic structures on a one-holed torus}
\label{sec: other structures}
Here we indicate how to alter the above proof of Theorem~\ref{main thm} to obtain Theorem~\ref{thm: other points}.
First, note that Definition~\ref{def: geom twist} admits a straightforward generalization when the finite-volume assumption is removed. 
Namely, the points $A_l$ and $A_r$, which had been the projections of the cusp of $\X$ to either side of $\gamma$, are replaced by the intersection points on $\gamma$ of the two simple orthogeodesics from the boundary parallel geodesic of $\cY$ to $\gamma$.

\begin{figure}[h]
\includegraphics[width=7cm]{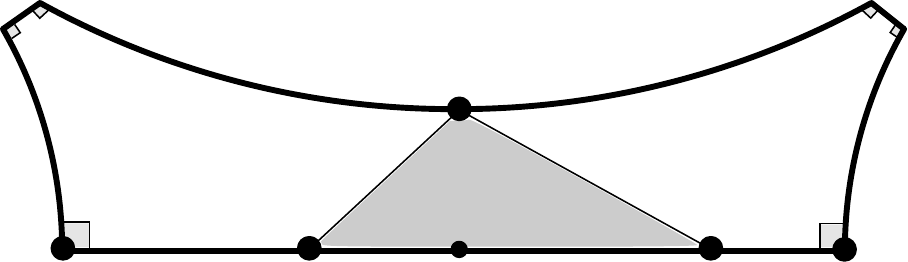}
\caption{The quadrilateral from Figure~\ref{pic: quadrilateral} becomes a right-angled hexagon when the boundary of $\cY$ has a geodesic representative.}
\label{pic: hexagon}
\end{figure}

Now we observe that the material from \S\ref{sec:compute geom} did not rely on the fact that the hyperbolic structure on $\cY$ is finite-volume; indeed, the quadrilateral $Q$ from Proposition~\ref{prop: picture of lift} becomes a right-angled hyperbolic hexagon (see Figure~\ref{pic: hexagon}), and \eqref{eq: ideal lambert} becomes
\[
\sinh \ell \cdot \sinh d(M,\widetilde{w_3}) = \cosh \frac\beta 2 ~,
\]
where $\beta$ is the length of the geodesic parallel to the boundary of $\cY$. We find that \eqref{eq: pythagoras} becomes
\[
\cosh \ell_1 = \cosh x \coth \ell \sqrt{ \tanh^2\ell+ \cosh^2 \frac\beta 2 \,\mathrm{sech}^2\,\ell }~,
\]
and \eqref{eq:x} becomes
\[
x= \cosh^{-1}\left( \frac{\tanh \ell \cosh \ell_1}{\tanh^2\ell+ \cosh^2 \frac\beta 2 \,\mathrm{sech}^2\,\ell}\right)~. 
\]
Observe that 
\[
\frac{\tanh \ell}{\tanh^2\ell+ \cosh^2 \frac\beta 2 \,\mathrm{sech}^2\,\ell} = 1+O\left( e^{-2\ell}\right)~,
\]
with an implied constant that depends on $\beta$, and hence
\[
x = \cosh^{-1}\left( \left( 1+ O\left( e^{-2\ell}\right)\right) \cosh \ell_1\right) = \ell_1 + O\left( e^{-2\ell}\right)~,
\]
as before.
Therefore we have
\[
\left| \tw_\cY(\gamma) -  \ell_\cY(\gamma_1) \right| \le Ce^{-\ell_\cY(\gamma)}~,
\]
for $C=C(\cY)>0$,
where $\gamma$ is the center of the ordered Farey triple $\gamma_2$, $\gamma$, $\gamma_1$ of simple geodesics on $\cY$.

As for the comparison between geometric and arithmetic twists as in \S\ref{sec: comparison}, we let
\[
\Psi_\cY\left(\frac pq\right) =  \frac {\ell_\cY\left( \frac pq\right)  }q~.
\]
Recall that Fock's $\Psi$ is continuous, positive, and convex by virtue of the fact that the hyperbolic length 
\[
L_\X(q,p) := \ell_\X\left(\gamma_{p/q}\right)
\]
extends continuously from primitive integral points to $H_1(\X,\bR)$ \cite[\S5]{McShane-Rivin}: first one extends linearly to $H_1(\X,\bQ)$, then one uses the strict convexity $L_\X(x+y)<L_\X(x)+L_\X(y)$, for $x$ and $y$ not collinear, to extend continuously to $H_1(\X,\bR)$.
Note that $L_\X$ is equal to zero only at the origin, and is continuous, homogeneous, and convex, so
\[
\Psi (x) = L_\X\left(1, x \right)
\]
is continuous, positive, and convex as well. (Note that because $\frac1q L_\X(q,p) = L_\X(1,p/q)$, this definition extends \eqref{eq:Fock}.)
The same argument applies verbatim to $\Psi_\cY$, which is therefore also continuous, positive, and convex. Therefore the argument in \S\ref{sec: comparison} demonstrates that \eqref{eq: control} holds for $\tau_\cY$ as well, with a constant $C=C(\cY)>0$. Therefore Theorem~\ref{thm: prescribe arith twist} implies that $\tau_\cY$ has dense graph as well.

\bigskip
\section{Explicit Computations}
\label{sec: explicit}
Here we point out that one can compute the relevant geometric quantities in \S\ref{sec:compute geom} rather explicitly for $\X$ as a function of a Markov \emph{triple}. 
Of course, this is unsatisfying, as MUC asserts that it is possible to deduce aspects of $\tw_\X(\gamma)$ from the corresponding Markov number alone.
Nonetheless, it is pleasant to observe the interaction of the hyperbolic geometry of $\X$ and the number theoretic aspects of $\cM$. The interested reader should consult \cite{Springborn1,Springborn2} for more detail on such connections.

First we point out that the absolute twist number $\tw_\X(\gamma)$ admits a surprisingly simple formula.

\begin{theorem}
\label{thm: explicit}
Suppose that $(\gamma_2,\gamma,\gamma_1)$ is a trio of oriented simple geodesics on $\X$ that pairwise intersect once so that either of the complementary triangles sees the curves in the given cyclic order, and suppose that the corresponding ordered Markov triple is $(n_1,n,n_2)$ (see Remark~\ref{rem: orientation switch}).
Then we have
\begin{equation}
\label{eq: twist}
\tw_\X(\gamma) = 2\cosh^{-1}\left( \frac{n_1}2 \sqrt{9-\frac4{n^2}}\right)
\end{equation}
\end{theorem}

Observe that we may conclude:

\begin{corollary}
\label{cor: nonzero}
For every $\gamma\in \cS$, $\tau_\X(\gamma)\ne 0$.
\end{corollary}

The proof of Theorem~\ref{thm: explicit} follows quickly from \eqref{eq: pythagoras}: 
\begin{figure}[h]
\begin{lpic}{rightTriangle(6cm)}
\lbl[]{-3,5;$f(\ell)$}
\lbl[]{10,-2;$x$}
\lbl[]{15,8;$\ell_1$}
\end{lpic}
\label{pic: right triangle}
\vspace{.5cm}
\end{figure}

The desired twist is $2x$, where 
\[
\cosh x = \frac{\cosh \ell_1}{\cosh\left( \sinh^{-1}\left( \frac 1{\sinh \ell}\right)\right)}= \frac{\cosh \ell_1}{\cosh \ell} \cdot \sinh \ell~,
\]
and $\cosh \ell$ and $\cosh \ell_1$ are given by $\frac32 n$ and $\frac 32n_1$, respectively.

More can be said about the geometric data associated to a simple closed geodesic on $\X$. 
In particular, the entire polygon $Q$ from Proposition~\ref{prop: picture of lift} can be made explicit.
We briefly sketch the computations:
By Proposition~\ref{prop: fractional markov}, there is a choice of lift $\tilde{\gamma}$ that is a hemisphere centered at $\frac32 + \frac kn$, with endpoints lying between $\frac kn$ and $3+\frac kn$. 
A tedious but straightforward computation demonstrates that
\[
\tilde{\gamma}(+\infty) = \frac{3+L}2+\frac kn \ , \ \ \text{ and} \ \ \ \tilde{\gamma}(-\infty) = \frac{3-L}2+\frac kn~,
\]
where $L=\sqrt{9-\frac4{n^2}}$ is the associated \emph{Lagrange number} \cite{Aigner}, and that the element $g\in \Gamma$ representing the curve $\gamma$ is given by
\begin{equation}
\label{eq:g}
g =
\begin{pmatrix}
3n+k & -3k- \frac{1+k^2}n \\
n & -k
\end{pmatrix}~.
\end{equation}
The point $A_l$ is evidently given by 
\[
A_l = \frac 32+\frac kn + \frac{iL}2 ~.
\]

Now we turn to the problem of locating $\widetilde{w_1}$. First observe that order two elements of $\PSL(2,\bZ)$ that normalize the cyclic subgroup generated by $h=\left( h_{ij}\right)$ have representative matrices 
\begin{equation}
\label{eq: matrix Weierstrass}
\begin{pmatrix}
a & b \\
c & -a 
\end{pmatrix}
\end{equation}
for $a,b,c\in\bZ$ so that $a^2+bc+1=0$, with fixed point $\frac{a+i}c$, and so that
\[
a \left( h_{11} - h_{22} \right) + b \, h_{12} + c \, h_{21} =0~.
\]
Replacing $h$ with \eqref{eq:g} and eliminating $b$ in this equation, we find 
\begin{equation}
\label{eq:quad form}
na^2-(3n+2k)ac+\left(3k+\frac{1+k^2}n\right)c^2 = -n~.
\end{equation}

\begin{remark}
The integral quadratic form on the lefthand-side of \eqref{eq:quad form} is yet another fundamental invariant of the Markov triple $(n_2,n,n_1)$, whose analysis goes back to Markov's pioneering work \cite{Markov}. One can check that this form is $\SL(2,\bZ)$-equivalent (via the substitution $(a,c)=(x+3y,y)$) to
\[
nx^2 + (3n-2k)xy+\left( -3k+ \frac{1+k^2}n\right)y^2~,
\]
often referred to as the \emph{Markov form} associated to the triple $n_1,n,n_2$ \cite[Ch.~2]{Cusick-Flahive}.
\end{remark}

Observe that $\widetilde{w_1}$ is the \emph{first} Weierstrass point along $\widetilde{\gamma}$ that follows $A_l$; therefore its real part $\Re(\widetilde{w_1})$ is minimum among the real parts $\frac ac$ of the fixed points of \eqref{eq: matrix Weierstrass}, i.e.~integral solutions to \eqref{eq:quad form}, subject to the constraint that $\frac ac \ge \Re(A_l)=\frac 32 +\frac kn$.

\begin{proposition}
\label{prop: neighbors}
The integral solutions $(a,c)$ to \eqref{eq:quad form} that satisfy $\frac kn \le \frac ac \le 3+\frac kn$ 
have the form $\{(a_i,c_i)\} \cup \{(a_i',c_i')\}$, where
\[
3+\frac{k_1}{n_1}=\frac {a_1}{c_1}  < \frac{a_2}{c_2} < \ldots 
\]
approaches $3 +\frac kn - \frac12 L$, where $(c_{i+1},c_i,n)$ is an ordered Markov triple, and $\frac{a_i}{c_i}$ is the fractional Markov label $\lambda_{\cF\cM}(H_i)$ on the central horoball $H_i$ from the triple, and
\[
\frac{k_2}{n_2}=\frac{a_1'}{c_1'} > \frac{a_2'}{c_2'} > \ldots
\]
approaches $\frac kn + \frac12 L$, where $(n,c_i',c_{i+1}')$ is an ordered Markov triple, with $\lambda_{\cF\cM}(H_i')=\frac{a_i'}{c_i'}$.
The solution minimizing $\frac ac$ subject to the constraint $\frac ac \ge \frac 32+ \frac kn$ is given by $(a_1,c_1)=(3n_1 + \frac {k n_1-n_2}n,n_1)$.
\end{proposition}

See Figures~\ref{pic:wlabels} and \ref{pic:ws}.

For the proof, note that the integral solutions to \eqref{eq:quad form} are given by $(a,c)$ so that $\tilde{\gamma}$ passes through a Weierstrass point with real part $\frac ac$. 
One can check that each of $(3n_1+k_1,n_1)$ and $(k_2,n_2)$ satisfy \eqref{eq:quad form}, which means that $3+\frac{k_1}{n_1}$ and $\frac{k_2}{n_2}$ are the real parts of vertical geodesics through Weierstrass points. 
It happens that the hyperbolic distance between the corresponding points is exactly $\ell$, the half-length of $\gamma$. Therefore all remaining Weierstrass points are obtained by application of powers of $g$ (as in \eqref{eq:g}) to $(3n_1+k_1,n_1)$ and $(k_2,n_2)$, and the Proposition can be deduced readily. We spare the reader the unnecessary details.

\begin{figure}
\begin{minipage}{.39\textwidth}
\centering
\begin{lpic}{wlabels(7cm)}
\LARGE
\lbl[]{15,50;$\frac{k_1}{n_1}$}
\lbl[]{100,50;$\frac{k_2}{n_2}$}
\lbl[]{58,26;$\frac kn$}
\lbl[]{22,13;$\frac{k_3}{n_3}$}
\lbl[]{94,13;$\frac{k_4}{n_4}$}
\end{lpic}
\vspace{1.cm}
\caption{Fractional Markov labels.}
\label{pic:wlabels}
\end{minipage}\hfill
\begin{minipage}{.6\textwidth}
\centering
\begin{lpic}{ws(9cm)}
\Large
\lbl[]{50,-8;$\frac kn$}
\lbl[]{102,-8;$\frac{k_2}{n_2}$}
\lbl[]{67,-8;$\frac{k_4}{n_4}$}
\lbl[]{166,-8;$3+\frac{k_1}{n_1}$}
\lbl[]{124,-8;$\frac32+\frac kn$}
\lbl[]{19,-8;$\frac{k_1}{n_1}$}
\lbl[]{35,-8;$\frac{k_3}{n_3}$}
\large
\lbl[]{40,-8;$\cdot$}
\lbl[]{42,-8;$\cdot$}
\lbl[]{44,-8;$\cdot$}
\lbl[]{62,-8;$\cdot$}
\lbl[]{60,-8;$\cdot$}
\lbl[]{58,-8;$\cdot$}
\LARGE
\lbl[]{80,65;$\widetilde{\gamma}$}
\end{lpic}
\vspace{.8cm}
\captionsetup{width=.8\linewidth}
\caption{Weierstrass points on $\widetilde{\gamma}$ can be determined from the fractional Markov labels.}
\label{pic:ws}
\end{minipage}
\end{figure}

Note that one can provide an alternative calculation of $\tw_\X(\gamma)$ using $\widetilde{w_1}$: By Proposition~\ref{prop: neighbors},
\[
\widetilde{w_1} = 3+ \frac kn - \frac{n_2}{n_1n} + \frac i{n_1}~.
\]

Now the distance $x$ (borrowing notation from \S\ref{sec:compute geom}) is given by the length of the arc of $\tilde{\gamma}$ that subtends the angle $\theta$, where
\[
\sin \theta = \frac{\Re(\widetilde{w_1}- A_l)}{\frac 12 L}
=\frac{ \left(3+\frac kn-\frac{n_2}{n_1 n}\right) - \left(\frac 32+\frac kn\right)}{\frac 12L} =\frac{\frac32 -\frac{n_2}{n_1n}}{\frac12 L}= \frac{3n_1 n-2n_2}{n_1nL}~.
\]
It is straightforward to check that the length of such an arc is given by 
\[
x= \tanh^{-1}\sin \theta =\frac 12 \log\left( \frac{n_1nL + 3n_1n-2n_2}{n_1nL -3n_1n+ 2n_2}\right)~,
\]
which the interested reader can check is equal to ~\eqref{eq: twist}.

\bigskip
\section{Pictures of $\tau_\cF$ and $\tau_\X$}

\begin{figure}[h]
\centering
\captionsetup{width=.8\linewidth}
\includegraphics[width=14cm]{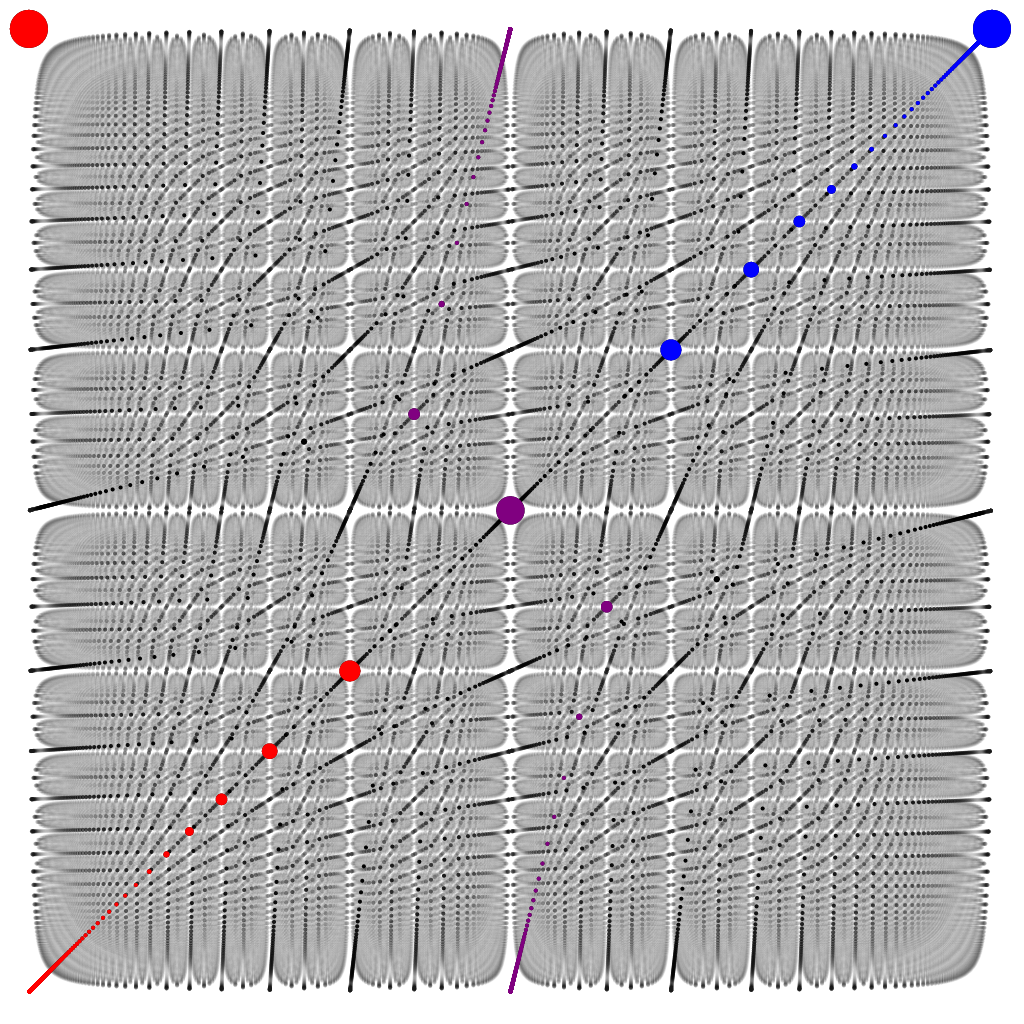}
\caption{The graph of $\tau_\cF$ on the $1216589$ 
points in $\bQ\cap[0,1]$ with denominators $\le 2000$. 
}
\label{pic: F twist graph}
\end{figure}

Images of the graphs of $\tau_\X$ and $\tau_\cF$, created with the aid of Mathematica \cite{Mathematica}, are depicted in Figures~\ref{pic: twist graph} and~\ref{pic: F twist graph}. 
In both graphs,
the point with $x$-coordinate $\frac pq$ is given a size and opacity that decays exponentially in the size of $q$;  
the larger red, blue, and purple points have $x$-coordinates $0$, $1$, and $\frac12$, respectively;
and points incident in the Farey graph to $\frac 01$ are red, to $\frac 11$ are blue, and to $\frac 12$ are purple.

\begin{problem}
\label{q: structure}
It appears that there is quite a bit of structure in the graphs of $\tau_\X$ and $\tau_\cF$. 
Formulate conjectures and explanations for these phenomena; 
e.g.~what can be said about the apparently `smooth' lines? what about the `gaps'? why do directions between $(0,0)$ and $(1,1)$ seem more prominent?
\end{problem}

\bigskip

\bibliographystyle{alpha}

\begin{thebibliography}{MGT21}

\bibitem[Aig15]{Aigner}
Martin Aigner.
\newblock {\em Markov's theorem and 100 years of the uniqueness conjecture}.
\newblock Springer, 2015.

\bibitem[Bar96]{Baragar2}
Arthur Baragar.
\newblock On the unicity conjecture for {M}arkoff numbers.
\newblock {\em Canad. Math. Bull.}, 39(1):3--9, 1996.

\bibitem[Bar18]{Baragar1}
Arthur Baragar.
\newblock Counting and geometry in number theory.
\newblock {\em CMS Notes}, 50(6):14--15, 2018.

\bibitem[BLS86]{Beardon-Lehner-Sheingorn}
Alan Beardon, Joseph Lehner, and Mark Sheingorn.
\newblock Closed geodesics on a riemann surface with application to the markov
  spectrum.
\newblock {\em Transactions of the American Mathematical Society},
  295(2):635--647, 1986.

\bibitem[Bon88]{Bonahon}
Francis Bonahon.
\newblock The geometry of teichm{\"u}ller space via geodesic currents.
\newblock {\em Inventiones mathematicae}, 92(1):139--162, 1988.

\bibitem[Bon91]{bonahon2}
Francis Bonahon.
\newblock Geodesic currents on negatively curved groups.
\newblock In {\em Arboreal group theory ({B}erkeley, {CA}, 1988)}, volume~19 of
  {\em Math. Sci. Res. Inst. Publ.}, pages 143--168. Springer, New York, 1991.

\bibitem[But98]{Button}
JO~Button.
\newblock The uniqueness of the prime {M}arkoff numbers.
\newblock {\em J. London Math. Soc. (2)}, 58(1):9--17, 1998.

\bibitem[CF89]{Cusick-Flahive}
Thomas~W Cusick and Mary~E Flahive.
\newblock {\em The Markoff and Lagrange spectra}.
\newblock Number~30. American Mathematical Soc., 1989.

\bibitem[Coh18]{Cohn}
Harvey Cohn.
\newblock Mathematical microcosm of geodesics, free groups, and markoff forms.
\newblock In {\em Classical and Quantum Models and Arithmetic Problems}, pages
  69--97. Routledge, 2018.

\bibitem[EMO02]{Eskin-Mozes-Oh}
Alex Eskin, Shahar Mozes, and Hee Oh.
\newblock Uniform exponential growth for linear groups.
\newblock {\em International Mathematics Research Notices},
  2002(31):1675--1683, 2002.

\bibitem[EPS20]{erlandsson-parlier-souto}
Viveka Erlandsson, Hugo Parlier, and Juan Souto.
\newblock Counting curves, and the stable length of currents.
\newblock {\em J. Eur. Math. Soc. (JEMS)}, 22(6):1675--1702, 2020.

\bibitem[EU20]{Erlandsson-Uyanik}
Viveka Erlandsson and Caglar Uyanik.
\newblock Length functions on currents and applications to dynamics and
  counting.
\newblock {\em In the Tradition of Thurston: Geometry and Topology}, pages
  423--458, 2020.

\bibitem[FM11]{Farb-Margalit}
Benson Farb and Dan Margalit.
\newblock {\em A primer on mapping class groups (pms-49)}, volume~41.
\newblock Princeton university press, 2011.

\bibitem[FN11]{Fenchel-Nielsen}
Werner Fenchel and Jakob Nielsen.
\newblock {\em Discontinuous groups of isometries in the hyperbolic plane},
  volume~29.
\newblock Walter de Gruyter, 2011.

\bibitem[Fro13]{Frobenius}
Ferdinand~Georg Frobenius.
\newblock \"{U}ber die markoffschen zahlen.
\newblock {\em S. Preuss. Akad. Wiss.}, XXVI:458--487, 1913.

\bibitem[FT18]{Fomin-Thurston}
Sergey Fomin and Dylan Thurston.
\newblock {\em Cluster algebras and triangulated surfaces Part II: Lambda
  lengths}, volume 255.
\newblock American Mathematical Society, 2018.

\bibitem[Gol03]{Goldman}
William~M Goldman.
\newblock The modular group action on real sl (2)--characters of a one-holed
  torus.
\newblock {\em Geometry \& Topology}, 7(1):443--486, 2003.

\bibitem[Haa86]{Haas}
Andrew Haas.
\newblock Diophantine approximation on hyperbolic {R}iemann surfaces.
\newblock {\em Acta Math.}, 156(1-2):33--82, 1986.

\bibitem[Hat88]{Hatcher-laminations}
Allen~E Hatcher.
\newblock Measured lamination spaces for surfaces, from the topological
  viewpoint.
\newblock {\em Topology and its Applications}, 30(1):63--88, 1988.

\bibitem[Hat22]{Hatcher}
Allen~E Hatcher.
\newblock {\em Topology of numbers}, volume 145.
\newblock American Mathematical Society, 2022.

\bibitem[Inc]{Mathematica}
Wolfram~Research{,} Inc.
\newblock Mathematica, {V}ersion 13.3.
\newblock Champaign, IL, 2023.

\bibitem[LS84]{Lehner-Sheingorn}
Joseph Lehner and Mark Sheingorn.
\newblock Simple closed geodesics on {$H\sp{+}/\Gamma (3)$} arise from the
  {M}arkov spectrum.
\newblock {\em Bull. Amer. Math. Soc. (N.S.)}, 11(2):359--362, 1984.

\bibitem[LT07]{Lang-Tan}
Mong~Lung Lang and Ser~Peow Tan.
\newblock A simple proof of the {M}arkoff conjecture for prime powers.
\newblock {\em Geom. Dedicata}, 129:15--22, 2007.

\bibitem[Mar80]{Markov}
Andrey Markov.
\newblock Sur les formes quadratiques binaires ind\'{e}finies.
\newblock {\em Math. Ann.}, 17(3):379--399, 1880.

\bibitem[Mat18]{Matheus}
Carlos Matheus.
\newblock The {L}agrange and {M}arkov spectra from the dynamical point of view.
\newblock In {\em Ergodic theory and dynamical systems in their interactions
  with arithmetics and combinatorics}, volume 2213 of {\em Lecture Notes in
  Math.}, pages 259--291. Springer, Cham, 2018.

\bibitem[McS]{McShane-new}
Greg McShane.
\newblock {E}isenstein integers and equilateral ideal triangles.
\newblock See \emph{https://macbuse.github.io/eisenstein.pdf}.

\bibitem[MGT21]{Martinez-Granado}
D\i~dac Mart{\'\i}nez-Granado and Dylan Thurston.
\newblock From curves to currents.
\newblock In {\em Forum of Mathematics, Sigma}, volume~9, page e77. Cambridge
  University Press, 2021.

\bibitem[Min99]{Minsky}
Yair Minsky.
\newblock The classification of punctured-torus groups.
\newblock {\em Annals of Mathematics}, pages 559--626, 1999.

\bibitem[MM99]{Masur-Minsky}
Howard Masur and Yair Minsky.
\newblock Geometry of the complex of curves i: Hyperbolicity.
\newblock {\em Inventiones Mathematicae}, 138(1):103--149, 1999.

\bibitem[MR95]{McShane-Rivin}
Greg McShane and Igor Rivin.
\newblock {A norm on homology of surfaces and counting simple geodesics}.
\newblock {\em International Mathematics Research Notices}, 1995(2):61--69, 01
  1995.

\bibitem[Ota90]{Otal}
Jean-Pierre Otal.
\newblock Le spectre marqu{\'e} des longueurs des surfaces {\`a} courbure
  n{\'e}gative.
\newblock {\em Annals of Mathematics}, 131(1):151--162, 1990.

\bibitem[Pen12]{Penner}
Robert~C Penner.
\newblock {\em Decorated {T}eichm\"{u}ller theory}.
\newblock QGM Master Class Series. European Mathematical Society (EMS),
  Z\"{u}rich, 2012.
\newblock With a foreword by Yuri I. Manin.

\bibitem[Sch06]{Schleimer}
Saul Schleimer.
\newblock Notes on the complex of curves.
\newblock {\em unpublished notes}, 2006.

\bibitem[Ser85]{Series}
Caroline Series.
\newblock The geometry of {M}arkoff numbers.
\newblock {\em Math. Intelligencer}, 7(3):20--29, 1985.

\bibitem[She93]{Sheingorn}
Mark Sheingorn.
\newblock Continued fractions and congruence subgroup geodesics.
\newblock In {\em Number theory with an emphasis on the {M}arkoff spectrum
  ({P}rovo, {UT}, 1991)}, volume 147 of {\em Lecture Notes in Pure and Appl.
  Math.}, pages 239--254. Dekker, New York, 1993.

\bibitem[Spr18]{Springborn1}
Boris Springborn.
\newblock The hyperbolic geometry of markov's theorem on diophantine
  approximation and quadratic forms.
\newblock {\em L'Enseignement Math{\'e}matique}, 63(3):333--373, 2018.

\bibitem[Spr22]{Springborn2}
Boris Springborn.
\newblock The worst approximable rational numbers.
\newblock {\em arXiv preprint arXiv:2209.15542}, 2022.

\bibitem[SS96]{Schmutz}
Paul Schmutz~Schaller.
\newblock Systoles of arithmetic surfaces and the {M}arkoff spectrum.
\newblock {\em Math. Ann.}, 305(1):191--203, 1996.

\bibitem[Thu22]{Thurston}
William~P Thurston.
\newblock {\em The Geometry and Topology of Three-Manifolds: With a Preface by
  Steven P. Kerckhoff}, volume~27.
\newblock American Mathematical Society, 2022.

\bibitem[Wol85]{Wolpert}
Scott Wolpert.
\newblock On the {W}eil-{P}etersson geometry of the moduli space of curves.
\newblock {\em Amer. J. Math.}, 107(4):969--997, 1985.

\bibitem[Zag82]{Zagier}
Don Zagier.
\newblock On the number of {M}arkoff numbers below a given bound.
\newblock {\em Mathematics of Computation}, 39(160):709--723, 1982.

\bibitem[Zha07]{Zhang}
Ying Zhang.
\newblock Congruence and uniqueness of certain {M}arkoff numbers.
\newblock {\em Acta Arith.}, 128(3):295--301, 2007.

\end{thebibliography}

\end{document}